\documentclass[11pt,twoside]{article}
\usepackage[margin=1.3in]{geometry}

\usepackage{amsfonts,amsmath,amssymb,amsthm}
\usepackage{bm,bbm,enumitem,latexsym,makeidx,mathrsfs,times}
\usepackage{cases,cite,color,xcolor}
\usepackage[colorlinks,citecolor=violet,linkcolor=teal]{hyperref}
\usepackage[title,titletoc,toc]{appendix}
\usepackage{esint}

\let\oldbibliography\thebibliography
\renewcommand{\thebibliography}[1]{\oldbibliography{#1}\setlength{\itemsep}{0pt}}

\makeindex
\newtheorem{theorem}{Theorem}[section]

\newtheorem{definition}{Definition}[section]
\newtheorem{lemma}{Lemma}[section]
\newtheorem{proposition}{Proposition}[section]
\newtheorem{remark}{Remark}[section]

\numberwithin{equation}{section}

\newcommand{\N}{\mathbb N}

\newcommand{\R}{\mathbb R}

\newcommand{\Sp}{\mathbb S}

\allowdisplaybreaks

\begin{document}
\title{\textbf{Symmetry, monotonicity, and asymptotics of singular solutions to semilinear elliptic equations}\bigskip}

\author{Xusheng Du\footnote{X. Du is supported by NSFC grant 12501146.} \quad and \quad Hui Yang\footnote{H. Yang is supported by NSFC grant 12301140.}}

\date{\today}

\maketitle

\begin{abstract}
In this paper, we study singular positive solutions to the semilinear elliptic equation
$$
- \Delta u = f(u) ~~~~~~ \textmd{in} ~ \Omega \setminus \Gamma,
$$
where $\Omega \subset \R^n$ is a bounded or unbounded domain, and $\Gamma \subset \Omega$ is a singular closed set with zero Newtonian capacity. When $\Omega = \R^n$ and $\Gamma \subset \{ x_1 = 0 \}$, we establish the symmetry of singular solutions with respect to the hyperplane $\{ x_1 = 0 \}$ and their monotonicity in the $x_1$-direction. For the case where $\Omega \subset \R^n$ and $\Gamma$ is a smooth closed manifold of dimension $\leq n - 2$, we show the asymptotic symmetry of singular solutions with respect to the normal direction of $\Gamma$. These results significantly improve those of Chen-Lin (Duke Math. J. 1995; Ann. Scuola Norm. Sup. Pisa Cl. Sci. 2001), Li (Invent. Math. 1996) and Sciunzi (J. Math. Pures Appl. 2017). Unlike their proofs, we employ an improved version of the moving sphere method, which also simplifies the analysis and extends its applicability.

\medskip


\medskip

\end{abstract}

\section{Introduction}\label{sec:Int}

In this paper, we are interested in singular positive solutions of the semilinear elliptic equation
\begin{equation}\label{eq:main}
- \Delta u = f(u) ~~~~~~ \textmd{in} ~ \Omega \setminus \Gamma,
\end{equation}
where $\Omega \subset \R^n$ is a bounded or unbounded domain, $n \geq 3$, and $\Gamma \subset \Omega$ is a closed set. The nonlinear term $f(t)$ is locally bounded on $(0, + \infty)$. Equation \eqref{eq:main}, also known as the nonlinear Poisson equation, frequently arises in geometry, physics, and various branches of applied sciences. For $f(t) = t^{- p}$ with $p > 0$, equation \eqref{eq:main} appears in certain problems in fluid mechanics, particularly in the study of pseudoplastic fluids \cite{NC80}. When $f(t) = t^p$ with $p > 0$, equation \eqref{eq:main} reduces to the Lane-Emden equation, a fundamental tool for analyzing the internal structure of stars and modeling pressure distribution in plasmas \cite{Sp56,Ch57}. In the special case $f(t) = t^\frac{n + 2}{n - 2}$, equation \eqref{eq:main} is the so-called Yamabe equation, which plays a central role in conformal geometry. The classical works of Schoen and Yau \cite{Sch88,SY88} on complete locally conformally flat manifolds have highlighted the significance of studying solutions to the Yamabe equation with a singular set $\Gamma$. The existence of singular solutions has been extensively investigated in \cite{Sch88,Pa94,MP96,CLin99,MP99,CD20} and the references therein, where singular solutions of the Lane-Emden equation ($1 < p < \frac{n + 2}{n - 2}$) have also played a crucial role. When $f(t) = - t^\frac{n + 2}{n - 2}$, equation \eqref{eq:main} with a singular set $\Gamma$ arises in the study of the Loewner-Nirenberg problem; see, e.g., \cite{LN74,Ve81,FM93,HJS} and the references therein. Furthermore, the combined nonlinearity $f(t) = \lambda t^p + \mu t^q$ also appears in Schr\"odinger equations that model nonlinear optics and Bose-Einstein condensates \cite{TVZ,MXZ}.

In this paper, we are mainly concerned with the symmetry, monotonicity, and asymptotics of singular positive solutions to equation \eqref{eq:main}. The singular set $\Gamma$ can not only be a discrete set of points but also a smooth submanifold. A solution $u$ has a non-removable singularity on $\Gamma$ if there exists at least one point $x_0 \in \Gamma$ such that
$$
\limsup_{x \in \Omega \setminus \Gamma \atop x \to x_0} u(x) = + \infty.
$$
To state our result, we recall that the Newtonian capacity of a compact set $\Gamma$ is defined as
$$
{\rm Cap} (\Gamma) = \inf \bigg\{ \int_{\R^n} |\nabla \varphi|^2 \,{\rm d}x : \varphi \in C_c^\infty (\R^n) ~ \textmd{and} ~ \varphi \geq 1 ~ \textmd{on} ~ \Gamma \bigg\}.
$$
For a general set $\Gamma$, the definition of Newtonian capacity can be seen in Section \ref{sec:Pre}. Note that if the $(n - 2)$-dimensional Hausdorff measure of $\Gamma$ is finite, then its Newtonian capacity ${\rm Cap} (\Gamma)$ vanishes. Throughout the paper, we always assume ${\rm Cap} (\Gamma) = 0$.

Our first focus is on the symmetry and monotonicity of global singular solutions to \eqref{eq:main} in the case where $\Omega = \R^n$ and $\Gamma$ is a closed subset of the hyperplane $\{ x_1 = 0 \}$ with vanishing Newtonian capacity.

\begin{theorem}\label{thm:sym-x1}
Suppose that $\Gamma \subset \{ x_1 = 0 \}$ is a closed set with ${\rm Cap} (\Gamma) = 0$, and $f : (0, + \infty) \to \R$ is locally bounded and satisfies
\begin{equation}\label{eq:mono}
t^{ - \frac{n + 2}{n - 2} } f(t) ~ \textmd{is nonincreasing on} ~ (0, + \infty).
\end{equation}
Let $u \in C^2 (\R^n \setminus \Gamma)$ be a positive solution of
\begin{equation}\label{eq:main-Rn}
- \Delta u = f(u) ~~~~~~ \textmd{in} ~ \R^n \setminus \Gamma.
\end{equation}
If $u$ has a non-removable singularity on $\Gamma$, then $u$ is symmetric with respect to the hyperplane $\{ x_1 = 0 \}$ and monotonically decreasing in the $x_1$-direction in $\{ x_1 > 0 \}$. In particular, when $\Gamma = \{ 0 \}$ and the origin is a non-removable singularity, $u$ is radially symmetric and monotonically decreasing about the origin.
\end{theorem}

\begin{remark}\label{rem:sym-x1}
We only assume that $f$ is locally bounded on $(0, + \infty)$, not on $[0, + \infty)$. This also includes the singular nonlinear term $f(t) = t^{- p}$. Moreover, we do not assume that $f$ is nonnegative. Hence, Theorem \ref{thm:sym-x1} can also be applied to the Loewner-Nirenberg problem ($f(t) = - t^\frac{n + 2}{n - 2}$).
\end{remark}

For the isolated singularity $\Gamma = \{ 0 \}$ and the nonlinearity $f(t) = t^p$ with $1 < p \leq \frac{n + 2}{n - 2}$, the symmetry of singular solutions was first proved by Caffarelli, Gidas and Spruck \cite{CGS} using the moving plane method. Chen and Lin \cite{CLin95,CLin01} later generalized this symmetry result to the case where $\Gamma$ is a $k$-dimensional subspace of $\R^n$ with $k \leq n - 2$. For a general singular set $\Gamma$ and nonlinearity $f$, Sciunzi \cite{Sci17} established the above symmetry and monotonicity under the following two conditions:
\begin{enumerate}[label = (\arabic*)]
\item $f$ is $C^1$ and convex (with $f(0) = 0$), and it holds that
\begin{equation}\label{eq:Sci-con}
f'(t) \leq C_f t^\frac{4}{n - 4} ~~~~~~ \textmd{for any} ~ t > 0.
\end{equation}

\item $\Gamma \subset \{ x_1 = 0 \} \cap B_{R_0}$ and $u \in L^\frac{2 n}{n - 2} (\R^n \setminus B_{R_0})$ for some $R_0 > 0$.
\end{enumerate}
Clearly, Theorem \ref{thm:sym-x1} improves the result of Sciunzi \cite{Sci17} in two key aspects. First, it does not require any regularity or convexity assumptions on $f$. Second, it removes the integrability condition $u \in L^\frac{2 n}{n - 2} (\R^n \setminus B_{R_0})$, which is not satisfied by certain classical singular solutions. For example, there exist Delaunay-type singular solutions satisfying $0 < c_1 \leq |x|^\frac{n - 2}{2} u(x) \leq c_2$ for all $x \in \R^n \setminus \{ 0 \}$ when $f(t) = t^\frac{n + 2}{n - 2}$ and $\Gamma = \{ 0 \}$. Moreover, although conditions \eqref{eq:mono} and \eqref{eq:Sci-con} are mutually independent, our condition \eqref{eq:mono} seems to cover more typical examples in physics and geometry. For instance, $f(t) = t^p$ with $p < \frac{n + 2}{n - 2}$ fails to meet \eqref{eq:Sci-con}. In the special case $f(t) = t^\frac{n + 2}{n - 2}$, Esposito, Farina and Sciunzi \cite{EFS} have also shown the symmetry without requiring the assumption $u \in L^\frac{2 n}{n - 2} (\R^n \setminus B_{R_0})$. However, their method does not seem to yield monotonicity.

When $\Omega$ is a bounded convex domain symmetric with respect to the hyperplane $\{ x_1 = 0 \}$, and the homogeneous Dirichlet condition is imposed on the boundary $\partial\Omega$, the symmetry of singular positive solutions of \eqref{eq:main} has also been extensively studied; see, e.g., \cite{CLN09,CLN12,CLN13,Sci17,EFS} and the references therein. These results generalize the celebrated symmetry theorem of Gidas-Ni-Nirenberg \cite{GNN}. Our second focus is on symmetry results over bounded domains without imposing any boundary conditions. This setting, which avoids boundary conditions, is of significant importance in many applications of conformal geometry. In this context, singular solutions exhibit asymptotic symmetry.

Assume that $\Gamma \subset \Omega$ is a smooth $k$-dimensional closed manifold with $k \leq n - 2$. Let $N$ be a tubular neighborhood of $\Gamma$ such that any point of $N$ can be uniquely expressed as the sum $z + v$ where $z \in \Gamma$ and $v \in (T_z \Gamma)^\perp$, the orthogonal complement of the tangent space of $\Gamma$ at $z$. Denote by $\Pi$ the orthogonal projection of $N$ onto $\Gamma$. For small $r > 0$ and $z \in \Gamma$, define
\begin{equation}\label{eq:Pir-1}
\Pi_r^{- 1} (z) = \{ x \in N : \Pi (x) = z,~ |x - z| = r \}.
\end{equation}
We establish the following asymptotic symmetry of solutions near the singular set $\Gamma$.

\begin{theorem}\label{thm:asym}
Let $\Omega$ be a bounded domain in $\R^n$, and $\Gamma \subset \Omega$ be a smooth $k$-dimensional closed manifold with $k \leq n - 2$. Let $N$ and $\Pi$ be described as above. Suppose that $f : (0, + \infty) \to (0, + \infty)$ is locally bounded and satisfies
\begin{equation}\label{asym-mono}
t^{ - \frac{n + 2}{n - 2} } f(t) ~ \textmd{is nonincreasing on} ~ (0, + \infty).
\end{equation}
Suppose further that $f$ satisfies the growth condition
\begin{equation}\label{eq:grow}
\lim_{t \to + \infty} t^{ - \frac{n}{n - 2} } f(t) = + \infty.
\end{equation}
Let $u \in C^2 (\Omega \setminus \Gamma)$ be a positive solution of \eqref{eq:main}. Then, for $x, y \in \Pi_r^{- 1} (z)$, we have
\begin{equation}\label{eq:asym}
u(x) = u(y) (1 + O(r)) ~~~~~~ \textmd{as} ~ r \to 0^+,
\end{equation}
where $O(r)$ is uniform for all $z \in \Gamma$.
\end{theorem}

For the case $\Gamma = \{ 0 \}$, the asymptotic symmetry was established by Caffarelli, Gidas and Spruck in their pioneering work \cite{CGS}. More precisely, they assumed that $f$ is a nonnegative locally Lipschitz function on $[0, + \infty)$ satisfying
\begin{enumerate}[label = (\arabic*)]
\item $f(t)$ is nondecreasing on $[0, + \infty)$, $f(0) = 0$;

\item $t^{ - \frac{n + 2}{n - 2} } f(t)$ is nonincreasing on $(0, + \infty)$;

\item $f(t) \geq c t^p$ for some $p \geq \frac{n}{n - 2}$ and for all sufficiently large $t$,
\end{enumerate}
and then proved that any solution is asymptotically symmetric near the origin, that is,
\begin{equation}\label{eq:asym-Ox}
u(x) = {\bar u} (|x|) (1 + O(|x|)) ~~~~~~ \textmd{as} ~ x \to 0,
\end{equation}
where ${\bar u} (r) = \fint_{ \Sp^{n - 1} } u(r \theta) \,{\rm d}S$ is the spherical average of $u$. Li \cite{CLi96} then simplified the method of Caffarelli-Gidas-Spruck and proved the asymptotic symmetry $u(x) = {\bar u} (|x|) (1 + o(1))$ as $x \to 0$, under the assumption that $f$ is a nonnegative locally Lipschitz function satisfying condition (2) only. In \cite{KMPS}, Korevaar, Mazzeo, Pacard and Schoen also developed an alternative approach to study the isolated singularities of the Yamabe equation. When $\Gamma$ is a smooth closed manifold of dimension $\leq n - 2$, Chen and Lin \cite{CLin95,CLin01} proved, under assumptions similar to those in Theorem \ref{thm:asym}, that any positive solution $u$ of \eqref{eq:main} satisfies
\begin{equation}\label{eq:asym-o1}
u(x) = u(y) (1 + o(1)) ~~~~~~ \textmd{for any} ~ x, y \in \Pi_r^{- 1} (z),
\end{equation}
where $o(1) \to 0$ as $r \to 0^+$. Our Theorem \ref{thm:asym} refines the result of Chen-Lin \cite{CLin95,CLin01} by enhancing the $o(1)$ remainder term to $O(|x|)$. This improvement is particularly useful for characterizing the precise asymptotic behavior of solutions near the singular set $\Gamma$. The proof by Chen and Lin \cite{CLin95,CLin01} is primarily based on the method of contradiction, which seems to make it difficult to directly derive the remainder term $O(|x|)$. On the other hand, due to the contradiction argument, Chen-Lin \cite{CLin01} only requires $t^{ - \frac{n + 2}{n - 2} } f(t)$ to be nonincreasing for large $t$.

When $\Gamma = \{ 0 \}$ is an isolated singularity, the condition \eqref{eq:grow} in Theorem \ref{thm:asym} can be removed, leading to the following asymptotic symmetry. This refinement strengthens the result of Li \cite{CLi96} by eliminating the Lipschitz assumption on $f$ and weakening the monotonicity assumption on $t^{ - \frac{n + 2}{n - 2} } f(t)$.

\begin{theorem}\label{thm:asym-0}
Let $\Omega \subset \R^n$ be a bounded domain containing the origin. Suppose that $f : (0, + \infty) \to (0, + \infty)$ is locally bounded and satisfies
\begin{equation}\label{eq:mono-0}
t^{ - \frac{n + 2}{n - 2} } f(t) ~ \textmd{is nonincreasing for} ~ t ~ \textmd{sufficiently large}.
\end{equation}
Let $u \in C^2 (\Omega \setminus \{ 0 \})$ be a positive solution of
\begin{equation}\label{eq:main-0}
- \Delta u = f(u) ~~~~~~ \textmd{in} ~ \Omega \setminus \{ 0 \}.
\end{equation}
Then the limit
\begin{equation}\label{eq:C0}
\lim_{x \to 0} |x|^{n - 2} u(x) = C_0 ~~~ \textmd{exists and is nonnegative}.
\end{equation}
Furthermore, if $C_0 = 0$, then $u$ is asymptotically symmetric as $x$ tends to $0$, i.e.,
\begin{equation}\label{eq:asym-0}
u(x) = {\bar u} (|x|) (1 + O(|x|)) ~~~~~~ \textmd{as} ~ x \to 0,
\end{equation}
where ${\bar u} (r) = \fint_{ \Sp^{n - 1} } u(r \theta) \,{\rm d}S$ is the spherical average of $u$.
\end{theorem}

\begin{remark}\label{rem:strong-sol}
Theorems \ref{thm:sym-x1}, \ref{thm:asym} and \ref{thm:asym-0} also hold for singular strong solutions of \eqref{eq:main}. A function $u$ is said to be a strong solution of \eqref{eq:main} if $u \in W_{\rm loc}^{2, p} (\Omega \setminus \Gamma)$ for any $p > 1$ and $u$ satisfies \eqref{eq:main} almost everywhere in $\Omega \setminus \Gamma$. By the Sobolev embedding theorems, $u \in C_{\rm loc}^{1, \alpha} (\Omega \setminus \Gamma)$ for any $\alpha \in (0, 1)$.
\end{remark}

The aforementioned articles \cite{CGS,CLin95,CLi96,CLin01,Sci17,EFS} mainly use the moving plane method to establish symmetry or asymptotic symmetry for singular solutions. We will employ an improved moving sphere method to prove Theorems \ref{thm:sym-x1}, \ref{thm:asym} and \ref{thm:asym-0}. The method of moving spheres, developed by Chen-Li \cite{CLi95}, Li-Zhu \cite{LZhu95} and Li-Zhang \cite{LZhang03}, has become a powerful tool in the study of various problems. For special nonlinear terms of type $f(t) = t^\frac{n + 2}{n - 2}$, singular solutions of Yamabe-type equations have been extensively investigated using this method, e.g., see \cite{Ado,CJSX,CdoS,DY,GKS,HXZ,JX,JY,Li06,XZ,Zhang02} and the references therein. Since our nonlinearity $f$ lacks Lipschitz regularity and is not of power type, while \eqref{eq:mono-0} holds only for large $t$, we have to refine some arguments in the process of moving spheres. The difficulties mainly include the following two aspects. On the one hand, when $f(t) = t^p$, the blow-up limiting equation of \eqref{eq:main} is
$$
- \Delta u = u^p ~~~~~~ \textmd{in} ~ \R^n,
$$
and all positive solutions to this equation have been completely classified (see \cite{CGS}). However, for a general $f(t)$, the blow-up limit does not satisfy any specific equation, and thus there are no classification results available. On the other hand, since \eqref{eq:mono-0} holds only for large $t$, the difference between $u$ and its Kelvin transform $u_{x, \bar\lambda}$
$$
u - u_{x, \bar\lambda} ~~~~~~ \textmd{in} ~ D \setminus (B_{\bar\lambda} (x) \cup \Gamma)
$$
is not superharmonic in a certain region $D$ like $\{ y : u(y) < u(y^{x, \bar\lambda}) \}$, and the maximum principle cannot be directly applied, where $\bar\lambda$ is the stopping position when moving spheres. Hence, we need to construct elaborate auxiliary functions in appropriate regions to apply the moving sphere method.

This paper is organized as follows. In Section \ref{sec:Pre}, we review the definition of Newtonian capacity and present several maximum principles that will be used later. In Section \ref{sec:Sym-x1}, we demonstrate the symmetry and monotonicity properties stated in Theorem \ref{thm:sym-x1}. In Section \ref{sec:Asym}, we establish the asymptotic symmetries in Theorems \ref{thm:asym} and \ref{thm:asym-0} using the improved method of moving spheres.

\section{Preliminaries}\label{sec:Pre}

Throughout the paper, we always assume the dimension $n \geq 3$. We denote by $B_r (x)$ the open ball in $\R^n$ of radius $r$ centered at $x$, and write $B_r (0)$ as $B_r$ for short. We also denote by $|B_1|$ the volume of $B_1$. First we review the definition of Newtonian capacity.

\begin{definition}\label{def:Cap}
Let $\Omega$ be an open set of $\R^n$ and $K$ be a compact subset of $\Omega$. The capacity of $K$ relative to $\Omega$ is defined by
$$
{\rm Cap} (K, \Omega) = \inf \bigg\{ \int_\Omega |\nabla u|^2 \,{\rm d}x : u \in C_c^\infty (\Omega),~ u \geq 1 ~ \textmd{on} ~ K \bigg\}.
$$
If $U \subset \Omega$ is open, define
$$
{\rm Cap} (U, \Omega) = \sup \{ {\rm Cap} (K, \Omega) : K \subset U ~ \textmd{compact} \},
$$
and for an arbitrary set $E \subset \Omega$,
$$
{\rm Cap} (E, \Omega) = \inf \{ {\rm Cap} (U, \Omega) : E \subset U \subset \Omega ~ \textmd{open} \}.
$$
\end{definition}

\begin{definition}\label{def:Cap0}
A set $E \subset \R^n$ is said to have capacity zero if
$$
{\rm Cap} (E \cap \Omega, \Omega) = 0 ~~~~~~ \textmd{for all open} ~ \Omega \subset \R^n.
$$
In this case, we write ${\rm Cap} (E) = 0$.
\end{definition}

The following lemma can be found in Lemma 2.9 of Heinonen-Kilpel\"ainen-Martio \cite{HKM}.

\begin{lemma}\label{lem:Cap0}
Suppose that $E$ is bounded and there is a bounded neighborhood $\Omega$ of $E$ with ${\rm Cap} (E, \Omega) = 0$. Then $E$ has capacity zero.
\end{lemma}

Notice that if $\Gamma \subset B_1$ is a smooth $k$-dimensional closed manifold with $k \leq n - 2$, then $\Gamma$ has capacity zero. More properties of capacity can be found in Heinonen-Kilpel\"ainen-Martio \cite{HKM} and Mal\'y-Ziemer \cite{MZ97}. We need the following maximum principle for superharmonic functions, which is a slight modification of Chen-Lin \cite[Lemma 2.1]{CLin95}.

\begin{lemma}\label{lem:suphar-int}
Let $\Omega$ be a smooth bounded domain in $\R^n$, and $\Gamma$ be a closed set in $\R^n$ with ${\rm Cap} (\Gamma) = 0$. Assume that $u \in C(\bar\Omega \setminus \Gamma)$ is superharmonic in $\Omega \setminus \Gamma$ in the distributional sense, and
$$
\inf_{\bar\Omega \setminus \Gamma} u > - \infty.
$$
Then
$$
u(y) \geq \inf_{\partial\Omega \setminus \Gamma} u
$$
for all $y \in \bar\Omega \setminus \Gamma$.
\end{lemma}

\begin{proof} Note that $\Gamma \cap \bar\Omega$ is a compact set with ${\rm Cap} (\Gamma \cap \bar\Omega) = 0$. Applying \cite[Lemma 2.1]{CLin95} to the function $(u - \inf_{\bar\Omega \setminus \Gamma} u)$ on $\bar\Omega \setminus (\Gamma \cap \bar\Omega)$, we can obtain the desired result.
\end{proof}

We also need the following maximum principle for second-order elliptic equations with a lower-order term, which can be found in Chen-Lin \cite[Lemma 4.4]{CLin95}.

\begin{lemma}\label{lem:supsol-int}
Let $\Omega$ be a smooth bounded domain in $\R^n$, and $\Gamma$ be a closed set in $\R^n$ with ${\rm Cap} (\Gamma) = 0$. Assume that $u \in C(\bar\Omega \setminus \Gamma)$ is a nonnegative function, $c(\cdot)$ is a measurable function in $\bar\Omega \setminus \Gamma$ satisfying
$$
\sup_{\bar\Omega \setminus \Gamma} c(y) < + \infty,
$$
and
$$
- \Delta u + c u \geq 0 ~~~~~~ \textmd{in} ~ \Omega \setminus \Gamma
$$
in the distributional sense. Then there exists a constant $0 < C \leq 1$ such that
$$
u(y) \geq C \inf_{\partial\Omega \setminus \Gamma} u
$$
for all $y \in \bar\Omega \setminus \Gamma$.
\end{lemma}

For $x \in \R^n$ and $\lambda > 0$, define the Kelvin transform of $u$ with respect to the sphere $\partial B_\lambda (x)$
\begin{equation}\label{eq:uK}
u_{x, \lambda} (y) = \bigg( \frac{\lambda}{|y - x|} \bigg)^{n - 2} u(y^{x, \lambda}),
\end{equation}
where
\begin{equation}\label{eq:yK}
y^{x, \lambda} := x + \frac{\lambda^2 (y - x)}{|y - x|^2}
\end{equation}
is the inversion point of $y$ with respect to the sphere $\partial B_\lambda (x)$. Then $u_{x, \lambda}$ satisfies
$$
\Delta u_{x, \lambda} (y) = \bigg( \frac{\lambda}{|y - x|} \bigg)^{n + 2} \Delta u(y^{x, \lambda}).
$$

Let $\Gamma \subset \{ y_1 = 0 \}$ be a closed set with ${\rm Cap} (\Gamma) = 0$. For $x = (x_1, 0, \dots, 0) \in \R^n \setminus \{ 0 \}$ and $0 < \lambda < |x_1|$, define
\begin{equation}\label{eq:GaK}
\Gamma^{x, \lambda} = \{ y^{x, \lambda} : y \in \Gamma \}
\end{equation}
as the inversion set of $\Gamma$ with respect to the sphere $\partial B_\lambda (x)$. Notice that the inversion set of the hyperplane $\{ y_1 = 0 \}$ with respect to the sphere $\partial B_\lambda (x)$ is the $(n - 1)$-dimensional sphere of radius $\frac{\lambda^2}{2 |x_1|}$ centered at $\big( x_1 - \frac{\lambda^2}{2 x_1}, 0, \dots, 0 \big)$, with the point $x$ removed.

\begin{lemma}\label{lem:Cap0-K}
Let $\Gamma \subset \{ y_1 = 0 \}$ be a closed set with ${\rm Cap} (\Gamma) = 0$. For any $x = (x_1, 0, \dots, 0) \in \R^n \setminus \{ 0 \}$ and $0 < \lambda < |x_1|$, we have ${\rm Cap} (\Gamma^{x, \lambda}) = 0$.
\end{lemma}

\begin{proof} The proof consists of two cases.

\medskip

{\it Case 1. $\Gamma$ is bounded.} In this case, $\Gamma$ is compact. Then there exist two constants $0 < \lambda < r < |x| < R < + \infty$ such that $\Gamma \subset B_R (x) \setminus \bar B_r (x)$. Through the inversion, $\Gamma^{x, \lambda} \subset B_{\lambda^2/r} (x) \setminus \bar B_{\lambda^2/R} (x)$. Because ${\rm Cap} (\Gamma) = 0$, by Definition \ref{def:Cap0} we have ${\rm Cap} (\Gamma, B_R (x) \setminus \bar B_r (x)) = 0$. Choose a sequence of functions $\{ u_i \} \subset C_c^\infty (B_R (x) \setminus \bar B_r (x))$ satisfying $u_i \geq 1$ on $\Gamma$ and
$$
\int_{B_R (x) \setminus \bar B_r (x)} |\nabla u_i|^2 \,{\rm d}y \to 0 ~~~~~~ \textmd{as} ~ i \to \infty.
$$
Set
$$
v_i (z) = \bigg\{
\aligned
& u_i (z^{x, \lambda}) ~~~~~~ && \textmd{if} ~ z \in B_{\lambda^2/r} (x) \setminus \bar B_{\lambda^2/R} (x), \\
& 0 ~~~~~~ && \textmd{otherwise}.
\endaligned
$$
Then $\{ v_i \} \subset C_c^\infty (B_{\lambda^2/r} (x) \setminus \bar B_{\lambda^2/R} (x))$ is a sequence of functions satisfying $v_i \geq 1$ on $\Gamma^{x, \lambda}$ and
$$
\aligned
\int_{B_{\lambda^2/r} (x) \setminus \bar B_{\lambda^2/R} (x)} |\nabla v_i (z)|^2 \,{\rm d}z & = \int_{B_R (x) \setminus \bar B_r (x)} |\nabla u_i (y)|^2 \bigg( \frac{\lambda}{|y - x|} \bigg)^{2 (n - 2)} \,{\rm d}y \\
& \leq \int_{B_R (x) \setminus \bar B_r (x)} |\nabla u_i (y)|^2 \,{\rm d}y \to 0
\endaligned
$$
as $i \to \infty$. Therefore, ${\rm Cap} (\Gamma^{x, \lambda}, B_{\lambda^2/r} (x) \setminus \bar B_{\lambda^2/R} (x)) = 0$. Since $B_{\lambda^2/r} (x) \setminus \bar B_{\lambda^2/R} (x)$ is bounded, by Lemma \ref{lem:Cap0} we obtain ${\rm Cap} (\Gamma^{x, \lambda}) = 0$.

\medskip

{\it Case 2. $\Gamma$ is unbounded.} For each positive integer $k$, $\Gamma \cap \bar B_k$ is compact and ${\rm Cap} (\Gamma \cap \bar B_k) = 0$. From Case 1, we know that ${\rm Cap} ((\Gamma \cap \bar B_k)^{x, \lambda}) = 0$. Now, for any open $\Omega \subset \R^n$,
$$
\aligned
{\rm Cap} (\Gamma^{x, \lambda} \cap \Omega, \Omega) & = {\rm Cap} \bigg( \bigg( \bigcup_{k = 1}^\infty (\Gamma \cap \bar B_k) \bigg)^{x, \lambda} \cap \Omega, \Omega \bigg) \\
& = {\rm Cap} \bigg( \bigcup_{k = 1}^\infty (\Gamma \cap \bar B_k)^{x, \lambda} \cap \Omega, \Omega \bigg) \\
& = \lim_{k \to \infty} {\rm Cap} ((\Gamma \cap \bar B_k)^{x, \lambda} \cap \Omega, \Omega) = 0.
\endaligned
$$
By Definition \ref{def:Cap0}, we conclude that ${\rm Cap} (\Gamma^{x, \lambda}) = 0$. The proof of Lemma \ref{lem:Cap0-K} is completed.
\end{proof}

Via the Kelvin transform, we can give the exterior region versions of Lemma \ref{lem:suphar-int} and Lemma \ref{lem:supsol-int} in the following.

\begin{lemma}\label{lem:suphar-ext}
Let $x = (x_1, 0, \dots, 0) \in \R^n \setminus \{ 0 \}$, $0 < \lambda < |x_1|$. Let $\Omega \subset \R^n \setminus \bar B_\lambda (x)$ be a smooth domain, and $\Gamma \subset \{ y_1 = 0 \}$ be a closed set with ${\rm Cap} (\Gamma) = 0$. Assume that $u \in C(\bar\Omega \setminus \Gamma)$ is superharmonic in $\Omega \setminus \Gamma$ in the distributional sense, and
$$
\inf_{y \in \bar\Omega \setminus \Gamma} |y - x|^{n - 2} u(y) > - \infty.
$$
Then
$$
u(y) \geq \bigg( \frac{\lambda}{|y - x|} \bigg)^{n - 2} \inf_{\partial\Omega \setminus \Gamma} \bigg( \bigg( \frac{|\cdot - x|}{\lambda} \bigg)^{n - 2} u \bigg)
$$
for all $y \in \bar\Omega \setminus \Gamma$.
\end{lemma}

\begin{proof} Let $u_{x, \lambda}$ be the Kelvin transform of $u$ as in \eqref{eq:uK}. Then $\Omega^{x, \lambda} \subset B_\lambda (x)$ is open, $\Gamma^{x, \lambda} \subset B_\lambda (x)$, and $u_{x, \lambda} \in C(\bar\Omega^{x, \lambda} \setminus (\Gamma^{x, \lambda} \cup \{ x \}))$ is a superharmonic function in $\Omega^{x, \lambda} \setminus \Gamma^{x, \lambda}$ in the distributional sense and bounded from below. By Lemma \ref{lem:Cap0-K}, we know that ${\rm Cap} (\Gamma^{x, \lambda}) = 0$, and thus ${\rm Cap} (\Gamma^{x, \lambda} \cup \{ x \}) = 0$. Since $\Gamma^{x, \lambda} \cup \{ x \}$ is closed, by Lemma \ref{lem:suphar-int}, we have
$$
u_{x, \lambda} (z) \geq \inf_{\partial\Omega^{x, \lambda} \setminus (\Gamma^{x, \lambda} \cup \{ x \})} u_{x, \lambda}
$$
for all $z \in \bar\Omega^{x, \lambda} \setminus (\Gamma^{x, \lambda} \cup \{ x \})$. By the definition of Kelvin transform, we arrive at the desired result.
\end{proof}

\begin{lemma}\label{lem:supsol-ext}
Let $x = (x_1, 0, \dots, 0) \in \R^n \setminus \{ 0 \}$, $0 < \lambda < |x_1|$. Let $\Omega \subset \R^n \setminus \bar B_\lambda (x)$ be a smooth domain, and $\Gamma \subset \{ y_1 = 0 \}$ be a closed set with ${\rm Cap} (\Gamma) = 0$. Assume that $u \in C(\bar\Omega \setminus \Gamma)$ is a nonnegative function, $c(\cdot)$ is a measurable function in $\bar\Omega \setminus \Gamma$ satisfying
$$
\sup_{y \in \bar\Omega \setminus \Gamma} |y - x|^4 c(y) < + \infty,
$$
and
$$
- \Delta u + c u \geq 0 ~~~~~~ \textmd{in} ~ \Omega \setminus \Gamma
$$
in the distributional sense. Then there exists a constant $0 < C \leq 1$ such that
$$
u(y) \geq C \bigg( \frac{\lambda}{|y - x|} \bigg)^{n - 2} \inf_{\partial\Omega \setminus \Gamma} \bigg( \bigg( \frac{|\cdot - x|}{\lambda} \bigg)^{n - 2} u \bigg)
$$
for all $y \in \bar\Omega \setminus \Gamma$.
\end{lemma}
The proof of Lemma \ref{lem:supsol-ext} is similar to that of Lemma \ref{lem:suphar-ext} and we omit it here.

\section{Symmetry and monotonicity for global singular solutions}\label{sec:Sym-x1}

In this section, we establish the symmetry and monotonicity properties in Theorem \ref{thm:sym-x1} via the method of moving spheres. Since the nonlinearity $f$ lacks both Lipschitz regularity and the nonnegativity assumption, the maximum principle presented in Lemma \ref{lem:supsol-ext} plays a crucial role in our analysis.

\begin{proof}[Proof of Theorem \ref{thm:sym-x1}] Without loss of generality, we assume that $0 \in \Gamma$ and
\begin{equation}\label{eq:singular-0}
\limsup_{y \in \R^n \setminus \Gamma \atop y \to 0} u(y) = + \infty.
\end{equation}
Denote by $\{ y_1 = 0 \}^\perp$ the orthogonal complement of the hyperplane $\{ y_1 = 0 \}$.

\medskip

\noindent{\bf Claim 1:} For every $x = (x_1, 0, \dots, 0) \in \{ y_1 = 0 \}^\perp \setminus \{ 0 \}$, there exists a real number $\lambda_2 \in (0, |x_1|)$ such that for any $0 < \lambda < \lambda_2$, we have
\begin{equation}\label{eq:claim1-g}
u(y) \geq u_{x, \lambda} (y) ~~~~~~ \forall\, y \in \R^n \setminus (B_\lambda (x) \cup \Gamma),
\end{equation}
where $u_{x, \lambda}$ is defined as in \eqref{eq:uK} and satisfies the equation
\begin{equation}\label{eq:uK-eq}
- \Delta u_{x, \lambda} (y) = \bigg( \frac{\lambda}{|y - x|} \bigg)^{n + 2} f(u(y^{x, \lambda})).
\end{equation}
The proof of Claim 1 consists of two steps.

\medskip

{\it Step 1.} We show that there exists $\lambda_1 \in (0, |x_1|)$ such that for any $0 < \lambda < \lambda_1$,
$$
u(y) \geq u_{x, \lambda} (y) ~~~~~~ \forall\, y \in B_{\lambda_1} (x) \setminus B_\lambda (x).
$$
Since $u \in C^2 (\R^n \setminus \Gamma)$ is positive, we can suppose
$$
|\nabla \ln u| \leq C_0 ~~~~~~ \textmd{in} ~ B_{|x_1|/2} (x)
$$
for some constant $C_0 > 0$. Then we have
\begin{equation}\label{eq:dru-g}
\aligned
\frac{ {\rm d} }{ {\rm d}r } \Big( r^\frac{n - 2}{2} u(x + r \theta) \Big) & = r^\frac{n - 4}{2} u(x + r \theta) \bigg( \frac{n - 2}{2} + \frac{\nabla u \cdot \theta}{u} r \bigg) \\
& \geq r^\frac{n - 4}{2} u(x + r \theta) \bigg( \frac{n - 2}{2} - C_0 r \bigg) > 0
\endaligned
\end{equation}
for all $0 < r < \lambda_1 := \min \big\{ \frac{n - 2}{2 C_0}, \frac{|x_1|}{2} \big\}$ and $\theta \in \Sp^{n - 1}$. For any $0 < \lambda < \lambda_1$ and $y \in B_{\lambda_1} (x) \setminus B_\lambda (x)$, let $\theta = \frac{y - x}{|y - x|}$, $r_1 = |y - x|$ and $r_2 = \frac{\lambda^2 r_1}{|y - x|^2}$. Using \eqref{eq:dru-g} we have
$$
r_1^\frac{n - 2}{2} u(x + r_1 \theta) \geq r_2^\frac{n - 2}{2} u(x + r_2 \theta).
$$
That is, for any $0 < \lambda < \lambda_1$,
\begin{equation}\label{eq:la1-g}
u(y) \geq u_{x, \lambda} (y) ~~~~~~ \forall\, y \in B_{\lambda_1} (x) \setminus B_\lambda (x).
\end{equation}

{\it Step 2.} We show that there exists $\lambda_2 \in (0, \lambda_1]$ such that \eqref{eq:claim1-g} holds for all $0 < \lambda < \lambda_2$.

\medskip

Let
\begin{equation}\label{eq:Om1-g}
\Omega_1 := \bigg\{ y \in \R^n \setminus (B_{\lambda_1} (x) \cup \Gamma) : u(y) < \bigg( \frac{\lambda_1}{|y - x|} \bigg)^{n - 2} \inf_{\partial B_{\lambda_1} (x)} u \bigg\}.
\end{equation}
By \eqref{eq:mono}, the monotone assumption of $f$,
$$
u^{ - \frac{n + 2}{n - 2} } f(u) \geq (\inf\nolimits_{\partial B_{\lambda_1} (x)} u)^{ - \frac{n + 2}{n - 2} } f(\inf\nolimits_{\partial B_{\lambda_1} (x)} u) ~~~~~~ \textmd{in} ~ \Omega_1.
$$
It follows that for $y \in \Omega_1$,
$$
- \Delta u = f(u) \geq \frac{f(\inf_{\partial B_{\lambda_1} (x)} u)}{ (\inf_{\partial B_{\lambda_1} (x)} u)^\frac{n + 2}{n - 2} } u^\frac{n + 2}{n - 2} \geq \min \bigg\{ 0, \frac{f(\inf_{\partial B_{\lambda_1} (x)} u)}{\inf_{\partial B_{\lambda_1} (x)} u} \bigg\} \bigg( \frac{\lambda_1}{|y - x|} \bigg)^4 u.
$$
Note that on $\partial\Omega_1 \setminus \Gamma$,
$$
\bigg( \frac{|y - x|}{\lambda_1} \bigg)^{n - 2} u(y) \geq \inf_{\partial B_{\lambda_1} (x)} u > 0.
$$
Since ${\rm Cap} (\Gamma) = 0$, by Lemma \ref{lem:supsol-ext}, there exists a constant $0 < C_1 \leq 1$ such that
$$
u(y) \geq C_1 \bigg( \frac{\lambda_1}{|y - x|} \bigg)^{n - 2} \inf_{\partial B_{\lambda_1} (x)} u > 0 ~~~~~~ \forall\, y \in \Omega_1.
$$
Obviously, the above inequality holds in $(\R^n \setminus (B_{\lambda_1} (x) \cup \Gamma)) \setminus \Omega_1$, and therefore
$$
u(y) \geq C_1 \bigg( \frac{\lambda_1}{|y - x|} \bigg)^{n - 2} \inf_{\partial B_{\lambda_1} (x)} u > 0 ~~~~~~ \forall\, y \in \R^n \setminus (B_{\lambda_1} (x) \cup \Gamma).
$$
Let
$$
\lambda_2 := \lambda_1 \bigg( \frac{C_1 \inf_{\partial B_{\lambda_1} (x)} u}{\sup_{B_{\lambda_1} (x)} u} \bigg)^\frac{1}{n - 2} \leq \lambda_1.
$$
Then for any $0 < \lambda < \lambda_2$ and $y \in \R^n \setminus (B_{\lambda_1} (x) \cup \Gamma)$, we have
$$
\aligned
u_{x, \lambda} (y) & = \bigg( \frac{\lambda}{|y - x|} \bigg)^{n - 2} u(y^{x, \lambda}) \\
& \leq \bigg( \frac{\lambda_2}{|y - x|} \bigg)^{n - 2} \sup_{B_{\lambda_1} (x)} u \\
& \leq C_1 \bigg( \frac{\lambda_1}{|y - x|} \bigg)^{n - 2} \inf_{\partial B_{\lambda_1} (x)} u \leq u(y).
\endaligned
$$
Inequality \eqref{eq:claim1-g} follows from \eqref{eq:la1-g} and the above. Claim 1 is proved.

Now, we can define
$$
\bar\lambda (x) := \sup \{ \mu \in (0, |x_1|) : u(y) \geq u_{x, \lambda} (y),~ \forall\, 0 < \lambda < \mu,~ y \in \R^n \setminus (B_\lambda (x) \cup \Gamma) \}.
$$
By Claim 1, $\bar\lambda (x)$ is well-defined and $0 < \bar\lambda (x) \leq |x_1|$.

\medskip

\noindent{\bf Claim 2:} $\bar\lambda (x) = |x_1|$ for all $x = (x_1, 0, \dots, 0) \in \{ y_1 = 0 \}^\perp \setminus \{ 0 \}$.

\medskip

Suppose $\bar\lambda (x) < |x_1|$ for some $x \in \{ y_1 = 0 \}^\perp \setminus \{ 0 \}$. For brevity, we denote $\bar\lambda (x) = \bar\lambda$ in the below. By the definition of $\bar\lambda$,
\begin{equation}\label{eq:bala-g}
u(y) \geq u_{x, \bar\lambda} (y) ~~~~~~ \forall\, y \in \R^n \setminus (B_{\bar\lambda} (x) \cup \Gamma).
\end{equation}
Furthermore, we want to prove
\begin{equation}\label{eq:strong-g}
u(y) > u_{x, \bar\lambda} (y) ~~~~~~ \forall\, y \in \R^n \setminus (\bar B_{\bar\lambda} (x) \cup \Gamma).
\end{equation}

Let
\begin{equation}\label{eq:Om2-g}
\Omega_2 := \Big\{ y \in \R^n \setminus (B_{\bar\lambda} (x) \cup \Gamma) : u(y) < \min \big\{ u(y^{x, \bar\lambda}), 2 u_{x, \bar\lambda} (y) \big\} \Big\}.
\end{equation}
Clearly, $\Omega_2$ is an open subset of $\R^n \setminus (\bar B_{\bar\lambda} (x) \cup \Gamma)$. By \eqref{eq:mono},
$$
u^{ - \frac{n + 2}{n - 2} } f(u) \geq u(y^{x, \bar\lambda})^{ - \frac{n + 2}{n - 2} } f(u(y^{x, \bar\lambda})) ~~~~~~ \textmd{in} ~ \Omega_2.
$$
It follows from \eqref{eq:bala-g} and the above that for $y \in \Omega_2$,
\begin{equation}\label{eq:sup-g1}
\aligned
- \Delta (u - u_{x, \bar\lambda}) & = f(u) - \bigg( \frac{\bar\lambda}{|y - x|} \bigg)^{n + 2} f(u(y^{x, \bar\lambda})) \\
& = u^{ - \frac{n + 2}{n - 2} } f(u) u^\frac{n + 2}{n - 2} - u(y^{x, \bar\lambda})^{ - \frac{n + 2}{n - 2} } f(u(y^{x, \bar\lambda})) u_{x, \bar\lambda}^\frac{n + 2}{n - 2} \\
& \geq u^{ - \frac{n + 2}{n - 2} } f(u) \Big( u^\frac{n + 2}{n - 2} - u_{x, \bar\lambda}^\frac{n + 2}{n - 2} \Big) \\
& \geq \min \Big\{ 0, u^{ - \frac{n + 2}{n - 2} } f(u) \Big\} \Big( u^\frac{n + 2}{n - 2} - u_{x, \bar\lambda}^\frac{n + 2}{n - 2} \Big) \\
& \geq \min \Big\{ 0, u^{ - \frac{n + 2}{n - 2} } f(u) \Big\} \frac{n + 2}{n - 2} u^\frac{4}{n - 2} (u - u_{x, \bar\lambda}),
\endaligned
\end{equation}
where we used the mean value theorem and \eqref{eq:bala-g} in the last inequality. On the other hand, in $\Omega_2$,
\begin{equation}\label{eq:uOm2-g}
u(y) \leq 2 u_{x, \bar\lambda} (y) = 2 \bigg( \frac{\bar\lambda}{|y - x|} \bigg)^{n - 2} u(y^{x, \bar\lambda}) \leq 2 \bigg( \frac{\bar\lambda}{|y - x|} \bigg)^{n - 2} \sup_{B_{\bar\lambda} (x)} u \leq 2 \sup_{B_{\bar\lambda} (x)} u.
\end{equation}
By \eqref{eq:mono} and \eqref{eq:uOm2-g},
$$
u^{ - \frac{n + 2}{n - 2} } f(u) \geq (2 \sup\nolimits_{B_{\bar\lambda} (x)} u)^{ - \frac{n + 2}{n - 2} } f(2 \sup\nolimits_{B_{\bar\lambda} (x)} u) ~~~~~~ \textmd{in} ~ \Omega_2.
$$
It follows from \eqref{eq:sup-g1}, \eqref{eq:uOm2-g} and the above that for $y \in \Omega_2$,
\begin{equation}\label{eq:sup-g2}
- \Delta (u - u_{x, \bar\lambda}) \geq \frac{n + 2}{n - 2} \min \bigg\{ 0, \frac{f(2 \sup_{B_{\bar\lambda} (x)} u)}{2 \sup_{B_{\bar\lambda} (x)} u} \bigg\} \bigg( \frac{\bar\lambda}{|y - x|} \bigg)^4 (u - u_{x, \bar\lambda}).
\end{equation}

With the help of the above analysis, we can prove \eqref{eq:strong-g} by contradiction. Suppose there exists $y_0 \in \R^n \setminus (\bar B_{\bar\lambda} (x) \cup \Gamma)$ such that $u(y_0) = u_{x, \bar\lambda} (y_0)$. Because $|y_0 - x| > \bar\lambda$, we get
$$
u(y_0^{x, \bar\lambda}) = \bigg( \frac{|y_0 - x|}{\bar\lambda} \bigg)^{n - 2} u_{x, \bar\lambda} (y_0) > u_{x, \bar\lambda} (y_0) = u(y_0).
$$
Then $y_0 \in \Omega_2$, and there is a small $\delta > 0$ such that $\bar B_\delta (y_0) \subset \Omega_2$. By \eqref{eq:bala-g}, \eqref{eq:sup-g2} and the strong maximum principle, $u \equiv u_{x, \bar\lambda}$ in $B_\delta (y_0)$. Repeating the same argument, we conclude that $\{ y \in \R^n \setminus (\bar B_{\bar\lambda} (x) \cup \Gamma) : u(y) = u_{x, \bar\lambda} (y) \}$ is a both open and closed nonempty subset of $\R^n \setminus (\bar B_{\bar\lambda} (x) \cup \Gamma)$. Since ${\rm Cap} (\Gamma) = 0$, the set $\R^n \setminus \Gamma$ is connected (see \cite[Lemma 2.46]{HKM}), so is $(\R^n \setminus \Gamma) \cap (\R^n \setminus \bar B_{\bar\lambda} (x)) = \R^n \setminus (\bar B_{\bar\lambda} (x) \cup \Gamma)$. Hence, $u \equiv u_{x, \bar\lambda}$ in $\R^n \setminus (\bar B_{\bar\lambda} (x) \cup \Gamma)$, and
$$
\limsup_{y \in \R^n \setminus \Gamma \atop y \to 0} u(y) = \limsup_{y \in \R^n \setminus \Gamma \atop y \to 0} u_{x, \bar\lambda} (y) = \bigg( \frac{\bar\lambda}{|x_1|} \bigg)^{n - 2} u \bigg( \bigg( 1 - \frac{\bar\lambda^2}{|x_1|^2} \bigg) x \bigg) < + \infty
$$
as $\bar\lambda < |x_1|$. This contradicts \eqref{eq:singular-0}. Inequality \eqref{eq:strong-g} is proved.

We also want to show
\begin{equation}\label{eq:hopf-g}
\partial_\nu (u - u_{x, \bar\lambda}) (y) > 0 ~~~~~~ \forall\, y \in \partial B_{\bar\lambda} (x),
\end{equation}
where $\nu$ is the unit outward normal vector on $\partial B_{\bar\lambda} (x)$. For $z_0 \in \partial B_{\bar\lambda} (x)$ satisfying $\partial_\nu (u - u_{x, \bar\lambda}) (z_0) < (n - 2) \bar\lambda^{- 1} u(z_0)$, by a direct calculation,
$$
(n - 2) \bar\lambda^{- 1} u(z_0) > \partial_\nu (u - u_{x, \bar\lambda}) (z_0) = (n - 2) \bar\lambda^{- 1} u(z_0) + 2 \partial_\nu u(z_0).
$$
Thus, $\partial_\nu u(z_0) < 0$ and
$$
\partial_\nu (u(\cdot^{x, \bar\lambda}) - u) (z_0) = - 2 \partial_\nu u(z_0) > 0.
$$
Recall that $u = u_{x, \bar\lambda} = u(\cdot^{x, \bar\lambda})$ on $\partial B_{\bar\lambda} (x)$, we can find a small $\delta > 0$ such that $B_\delta (z_0) \setminus \bar B_{\bar\lambda} (x) \subset \Omega_2$. By \eqref{eq:strong-g}, \eqref{eq:sup-g2} and the Hopf lemma, $\partial_\nu (u - u_{x, \bar\lambda}) (z_0) > 0$. Inequality \eqref{eq:hopf-g} is established.

By \eqref{eq:hopf-g} and the compactness of $\partial B_{\bar\lambda} (x)$, we have
$$
\partial_\nu (u - u_{x, \bar\lambda}) \big|_{\partial B_{\bar\lambda} (x)} \geq b > 0.
$$
By the continuity of $u$ and $\nabla u$, there exists $R \in (\bar\lambda, |x_1|)$ such that for any $\bar\lambda \leq \lambda \leq r \leq R$,
$$
\partial_\nu (u - u_{x, \lambda}) \big|_{\partial B_r (x)} \geq \frac{b}{2} > 0.
$$
Consequently, since $u = u_{x, \lambda}$ on $\partial B_\lambda (x)$, we have for any $\bar\lambda \leq \lambda < R$,
\begin{equation}\label{eq:BR-g}
u(y) > u_{x, \lambda} (y) ~~~~~~ \forall\, y \in B_R (x) \setminus \bar B_\lambda (x).
\end{equation}

To apply Lemma \ref{lem:supsol-ext} to $(u - u_{x, \bar\lambda})$ on $\Omega_2 \setminus \bar B_R (x)$, we need to estimate $(u - u_{x, \bar\lambda})$ on $\partial (\Omega_2 \setminus \bar B_R (x)) \setminus \Gamma$. If $y \in \partial B_R (x)$, then by \eqref{eq:strong-g},
$$
(u - u_{x, \bar\lambda}) (y) \geq \inf_{\partial B_R (x)} (u - u_{x, \bar\lambda}) > 0.
$$
If $y \in \R^n \setminus (\bar B_R (x) \cup \Gamma)$ and $u(y) = u(y^{x, \bar\lambda})$, then
$$
\aligned
\bigg( \frac{|y - x|}{R} \bigg)^{n - 2} (u - u_{x, \bar\lambda}) (y) & = \bigg( \frac{|y - x|}{R} \bigg)^{n - 2} \bigg[ 1 - \bigg( \frac{\bar\lambda}{|y - x|} \bigg)^{n - 2} \bigg] u(y^{x, \bar\lambda}) \\
& \geq \bigg[ 1 - \bigg( \frac{\bar\lambda}{R} \bigg)^{n - 2} \bigg] \inf_{B_{\bar\lambda} (x)} u > 0.
\endaligned
$$
If $y \in \R^n \setminus (\bar B_R (x) \cup \Gamma)$ and $u(y) = 2 u_{x, \bar\lambda} (y)$, then
$$
\bigg( \frac{|y - x|}{R} \bigg)^{n - 2} (u - u_{x, \bar\lambda}) (y) = \bigg( \frac{\bar\lambda}{R} \bigg)^{n - 2} u(y^{x, \bar\lambda}) \geq \bigg( \frac{\bar\lambda}{R} \bigg)^{n - 2} \inf_{B_{\bar\lambda} (x)} u > 0.
$$
Now, by Lemma \ref{lem:supsol-ext},
\begin{equation}\label{eq:uC2G-g}
(u - u_{x, \bar\lambda}) (y) \geq C_2 \bigg( \frac{R}{|y - x|} \bigg)^{n - 2} > 0 ~~~~~~ \forall\, y \in \Omega_2 \setminus B_R (x)
\end{equation}
for some constant
$$
0 < C_2 \leq \min \bigg\{ \inf_{\partial B_R (x)} (u - u_{x, \bar\lambda}), \bigg[ 1 - \bigg( \frac{\bar\lambda}{R} \bigg)^{n - 2} \bigg] \inf_{B_{\bar\lambda} (x)} u, \bigg( \frac{\bar\lambda}{R} \bigg)^{n - 2} \inf_{B_{\bar\lambda} (x)} u \bigg\}.
$$
Obviously, \eqref{eq:uC2G-g} holds automatically in $(\R^n \setminus (B_R (x) \cup \Gamma)) \setminus \Omega_2$, and therefore
$$
(u - u_{x, \bar\lambda}) (y) \geq C_2 \bigg( \frac{R}{|y - x|} \bigg)^{n - 2} > 0 ~~~~~~ \forall\, y \in \R^n \setminus (B_R (x) \cup \Gamma).
$$
It follows that for $y \in \R^n \setminus (B_R (x) \cup \Gamma)$,
\begin{equation}\label{eq:uC2G-g2}
(u - u_{x, \lambda}) (y) \geq C_2 \bigg( \frac{R}{|y - x|} \bigg)^{n - 2} - (u_{x, \lambda} - u_{x, \bar\lambda}) (y).
\end{equation}
By the uniform continuity of $u$ on $\bar B_R (x)$, there exists $0 < \varepsilon \leq R - \bar\lambda$ such that for all $\bar\lambda \leq \lambda < \bar\lambda + \varepsilon$ and $y \in \R^n \setminus (B_R (x) \cup \Gamma)$,
$$
\bigg| \lambda^{n - 2} u \bigg( x + \frac{\lambda^2 (y - x)}{|y - x|^2} \bigg) - \bar\lambda^{n - 2} u \bigg( x + \frac{\bar\lambda^2 (y - x)}{|y - x|^2} \bigg) \bigg| \leq \frac{ C_2 R^{n - 2} }{2}.
$$
This together with \eqref{eq:uC2G-g2} implies that for all $\bar\lambda \leq \lambda < \bar\lambda + \varepsilon$,
$$
u(y) > u_{x, \lambda} (y) ~~~~~~ \forall\, y \in \R^n \setminus (B_R (x) \cup \Gamma).
$$
This and \eqref{eq:BR-g} contradict the definition of $\bar\lambda$. The proof of Claim 2 is completed.

Thus, we have shown that for every $x = (x_1, 0, \dots, 0) \in \{ y_1 = 0 \}^\perp \setminus \{ 0 \}$,
\begin{equation}\label{eq:moving-g}
u(y) \geq u_{x, \lambda} (y) ~~~~~~ \forall\, 0 < \lambda < |x_1|,~ y \in \R^n \setminus (B_\lambda (x) \cup \Gamma).
\end{equation}
Denote $e_1 = (1, 0, \dots, 0) \in \R^n$. For any $a > 0$, $y = (y_1, y') \in \R^n \setminus \{ y_1 = 0 \}$ with $y_1 < a$, and for any $R > a$, we have, by \eqref{eq:moving-g} with $x = R e_1$ and $\lambda = R - a$,
$$
\aligned
u(y_1, y') & \geq u_{x, \lambda} (y_1, y') \\
& = \bigg( \frac{R - a}{|y - R e_1|} \bigg)^{n - 2} u \bigg( R + \frac{(R - a)^2 (y_1 - R)}{|y - R e_1|^2}, \frac{(R - a)^2 y'}{|y - R e_1|^2} \bigg) \\
& = \bigg( \frac{R - a}{|y - R e_1|} \bigg)^{n - 2} u \bigg( \frac{R^2 (2 a - y_1) + R (|y|^2 - 2 a y_1 - a^2) + a^2 y_1}{|y - R e_1|^2}, \frac{(R - a)^2 y'}{|y - R e_1|^2} \bigg).
\endaligned
$$
Here we used the simple fact $|y - x| = |y - R e_1| \geq |y_1 - R| \geq R - a = \lambda$. Sending $R$ to infinity in the above, we obtain
\begin{equation}\label{eq:mono-y+}
u(y_1, y') \geq u(2 a - y_1, y') ~~~~~~ \forall\, a > 0,~ y \in \R^n \setminus \{ y_1 = 0 \} ~ \textmd{with} ~ y_1 < a.
\end{equation}
Similarly,
\begin{equation}\label{eq:mono-y-}
u(y_1, y') \geq u(2 a - y_1, y') ~~~~~~ \forall\, a < 0,~ y \in \R^n \setminus \{ y_1 = 0 \} ~ \textmd{with} ~ y_1 > a.
\end{equation}
Since $a$ is arbitrary, \eqref{eq:mono-y+} and \eqref{eq:mono-y-} give the symmetry of $u$ with respect to the hyperplane $\{ y_1 = 0 \}$. \eqref{eq:mono-y+} also implies that $u$ is monotonically decreasing in the $y_1$-direction in $\{ y_1 > 0 \}$. The proof of Theorem \ref{thm:sym-x1} is completed.
\end{proof}

\section{Asymptotic symmetry for local singular solutions}\label{sec:Asym}

In this section, we refine the method of moving spheres to establish asymptotic symmetry results stated in Theorems \ref{thm:asym} and \ref{thm:asym-0}. First, we derive an upper bound estimate for positive solutions of \eqref{eq:main} near the singular set $\Gamma$ when $\Gamma \subset \Omega$ is a smooth closed manifold of dimension $\leq n - 2$. A similar upper bound estimate was obtained by Chen and Lin \cite{CLin01} via the moving plane method. Here, we establish this estimate by employing an improved moving sphere method, which is simpler. This refined approach will also be applied in the proofs of Theorems \ref{thm:asym} and \ref{thm:asym-0}.

\begin{proposition}\label{prop:blowup}
Let $\Omega$ be a bounded domain in $\R^n$, and $\Gamma \subset \Omega$ be a smooth $k$-dimensional closed manifold with $k \leq n - 2$. Suppose that $f : (0, + \infty) \to (0, + \infty)$ is locally bounded and satisfies
$$
t^{ - \frac{n + 2}{n - 2} } f(t) ~ \textmd{is nonincreasing for} ~ t ~ \textmd{sufficiently large}.
$$
Let $u \in C^2 (\Omega \setminus \Gamma)$ be a positive solution of \eqref{eq:main}. Then there exists a constant $C > 0$ such that
\begin{equation}\label{eq:blowup}
\frac{f(u(x))}{u(x)} \leq C d(x, \Gamma)^{- 2}
\end{equation}
for all $x$ near $\Gamma$, where $d(x, \Gamma)$ is the distance between $x$ and $\Gamma$.
\end{proposition}

\begin{proof}[Proof of Proposition \ref{prop:blowup}] For simplicity, we assume that $\Omega = B_2$, $\Gamma \subset B_{1/2}$, $u \in C^2 (\bar B_2 \setminus \Gamma)$ and $f$ is continuous on $(0, + \infty)$. Notice that there is a constant $t_0 \geq 0$ such that
\begin{equation}\label{eq:t0-b}
t^{ - \frac{n + 2}{n - 2} } f(t) ~ \textmd{is nonincreasing for} ~ t > t_0.
\end{equation}
We prove \eqref{eq:blowup} by contradiction. Suppose there exists a sequence $\{ x_j \} \subset B_{1/2} \setminus \Gamma$ such that
\begin{equation}\label{eq:dj2fuxj}
d_j^2 \frac{f(u(x_j))}{u(x_j)} = \sup \bigg\{ d(x, \Gamma)^2 \frac{f(u(x))}{u(x)} : x \in \bar B_2 \setminus \Gamma,~ d(x, \Gamma) \geq d_j \bigg\},
\end{equation}
where
$$
d_j := d(x_j, \Gamma) \to 0 ~~~~~~ \textmd{as} ~ j \to \infty,
$$
but
$$
d_j^2 \frac{f(u(x_j))}{u(x_j)} \to + \infty ~~~~~~ \textmd{as} ~ j \to \infty.
$$

Consider
$$
v_j (x, z) := \bigg( \frac{d_j}{2} - |x - z| \bigg)^2 \frac{f(u(z))}{u(z)} ~~~~~~ \textmd{for} ~ d(x, \Gamma) = d_j,~ z \in \bar B_{d_j/2} (x).
$$
Since $u$ is positive and continuous in $\bar B_2 \setminus \Gamma$, we can find points $\bar x_j, z_j$ such that $d(\bar x_j, \Gamma) = d_j$ and $z_j \in B_{d_j/2} (\bar x_j)$ satisfying
\begin{equation}\label{eq:max-vj}
v_j (\bar x_j, z_j) = \max_{d(x, \Gamma) = d_j \atop z \in \bar B_{d_j/2} (x)} v_j (x, z) \geq v_j (x_j, x_j) = \frac{d_j^2}{4} \cdot \frac{f(u(x_j))}{u(x_j)} > 0.
\end{equation}
Let $r_j := d(z_j, \Gamma)$. Then
\begin{equation}\label{eq:rj-dj}
r_j \geq d(\bar x_j, \Gamma) - |\bar x_j - z_j| > d_j - \frac{d_j}{2} \geq \frac{d_j}{2}
\end{equation}
and
$$
r_j \leq d(\bar x_j, \Gamma) + |\bar x_j - z_j| < d_j + \frac{d_j}{2} \leq \frac{3 d_j}{2} \ll d(\Gamma, \partial B_{1/2})
$$
for large $j$. By \eqref{eq:dj2fuxj}, \eqref{eq:max-vj} and the definition of $v_j$, we have
\begin{equation}\label{eq:fuzj}
\frac{f(u(z_j))}{u(z_j)} = \sup \bigg\{ \frac{f(u(x))}{u(x)} : x \in \bar B_2 \setminus \Gamma,~ d(x, \Gamma) \geq r_j \bigg\}.
\end{equation}
Let $2 \mu_j := \frac{d_j}{2} - |\bar x_j - z_j|$. Then
\begin{equation}\label{eq:2muj}
0 < 2 \mu_j \leq \frac{d_j}{2} ~~~~~~ \textmd{and} ~~~~~~ \frac{d_j}{2} - |\bar x_j - z| \geq \mu_j ~~~ \forall\, z \in \bar B_{\mu_j} (z_j).
\end{equation}
By the definition of $v_j$, we have
$$
4 \mu_j^2 \cdot \frac{f(u(z_j))}{u(z_j)} = v_j (\bar x_j, z_j) \geq v_j (\bar x_j, z) \geq \mu_j^2 \cdot \frac{f(u(z))}{u(z)} ~~~~~~ \forall\, z \in \bar B_{\mu_j} (z_j).
$$
Hence, we have
\begin{equation}\label{eq:4fuzj}
4 \cdot \frac{f(u(z_j))}{u(z_j)} \geq \frac{f(u(z))}{u(z)} ~~~~~~ \forall\, z \in \bar B_{\mu_j} (z_j).
\end{equation}
We also have
\begin{equation}\label{eq:vj-xjxj}
4 \mu_j^2 \cdot \frac{f(u(z_j))}{u(z_j)} = v_j (\bar x_j, z_j) \geq v_j (x_j, x_j) = \frac{d_j^2}{4} \cdot \frac{f(u(x_j))}{u(x_j)} \to + \infty ~~~~~~ \textmd{as} ~ j \to \infty,
\end{equation}
and, by \eqref{eq:rj-dj} and \eqref{eq:2muj},
\begin{equation}\label{eq:rj2fuzj}
r_j^2 \cdot \frac{f(u(z_j))}{u(z_j)} \to + \infty ~~~~~~ \textmd{as} ~ j \to \infty.
\end{equation}
Because $f$ is positive, we know that $u$ is superharmonic in $B_2 \setminus \Gamma$. By Lemma \ref{lem:suphar-int}, $u \geq \inf_{\partial B_2} u > 0$ in $\bar B_2 \setminus \Gamma$. This together with \eqref{eq:rj2fuzj} and the local boundness of $f$ implies that
\begin{equation}\label{eq:uzj}
u(z_j) \to + \infty ~~~~~~ \textmd{as} ~ j \to \infty.
\end{equation}

Now, define
\begin{equation}\label{eq:wj}
w_j (y) = \frac{1}{u(z_j)} u \bigg( z_j + \bigg( \frac{u(z_j)}{f(u(z_j))} \bigg)^\frac{1}{2} y \bigg) ~~~~~~ \textmd{for} ~ y \in \Sigma_j,
\end{equation}
where
$$
\Sigma_j := \bigg\{ y \in \R^n : z_j + \bigg( \frac{u(z_j)}{f(u(z_j))} \bigg)^\frac{1}{2} y \in B_2 \setminus \Gamma \bigg\}.
$$
Then $w_j$ satisfies $w_j (0) = 1$ and
\begin{equation}\label{eq:wj-eq}
- \Delta w_j = \frac{f(u(z_j) w_j)}{f(u(z_j))} ~~~~~~ \textmd{in} ~ \Sigma_j.
\end{equation}
Moreover, we can rewrite \eqref{eq:wj-eq} further as
$$
- \Delta w_j = \bigg( \frac{f(u(z_j) w_j)}{f(u(z_j)) w_j} \bigg) w_j ~~~~~~ \textmd{in} ~ \Sigma_j.
$$
It follows from \eqref{eq:4fuzj} and \eqref{eq:vj-xjxj} that
$$
0 \leq \frac{f(u(z_j) w_j)}{f(u(z_j)) w_j} \leq 4 ~~~~~~ \textmd{in} ~ \bar B_{R_j},
$$
where
$$
R_j := \mu_j \bigg( \frac{f(u(z_j))}{u(z_j)} \bigg)^\frac{1}{2} \to + \infty ~~~~~~ \textmd{as} ~ j \to \infty.
$$
By the Harnack inequality, for any $R > 0$, there exists a constant $C = C(n, R) > 1$ such that
\begin{equation}\label{eq:Harnack}
1 = w_j (0) \leq \sup_{B_R} w_j \leq C \inf_{B_R} w_j \leq C
\end{equation}
for large $j$. Thus, $\{ w_j \}$ is uniformly bounded in any compact set of $\R^n$. By the standard elliptic interior estimates, $\{ w_j \}$ is bounded in $C_{\rm loc}^{1, \alpha} (\R^n)$ for any $0 < \alpha < 1$, and there exists a nonnegative function $w$ such that, after passing to a subsequence, $w_j \to w$ in $C_{\rm loc}^{1, \alpha} (\R^n)$ for any $0 < \alpha < 1$. Clearly, $w(0) = 1$. By \eqref{eq:Harnack}, we also know that $w > 0$ in $\R^n$.

\medskip

\noindent{\bf Claim 1:} We can assume that $w \not\equiv 1$ in $\R^n$. Otherwise, there is a sequence of $\{ \tilde z_j \} \subset B_{1/2} \setminus \Gamma$ such that if we replace $z_j$ by $\tilde z_j$ in \eqref{eq:wj} and denote the new $w_j$ by $\tilde w_j$, then \eqref{eq:rj2fuzj} holds, and $\tilde w_j$ converges to a nonconstant positive function $\tilde w$ in $C_{\rm loc}^{1, \alpha} (\R^n)$ for any $0 < \alpha < 1$.

\medskip

If $w \equiv 1$ in $\R^n$. Let $N \subset B_{1/2}$ be a tubular neighborhood of $\Gamma$ such that any point of $N$ can be uniquely expressed as the sum $x + v$ where $x \in \Gamma$ and $v \in (T_x \Gamma)^\perp$, the orthogonal complement of the tangent space of $\Gamma$ at $x$. Denote by $\Pi$ the orthogonal projection of $N$ onto $\Gamma$. Then $z_j \in N \setminus \Gamma$ for large $j$. By \eqref{eq:uzj} and the continuity of $u$, there exists
$$
\tilde z_j = \Pi (z_j) + t_j (z_j - \Pi (z_j)) \in N ~~~~~~ \textmd{for some} ~ t_j \in (1, + \infty)
$$
such that
\begin{equation}\label{eq:utzj}
u(\tilde z_j) = \frac{1}{2} u(z_j).
\end{equation}
Hence,
\begin{equation}\label{eq:dtzj}
d(\tilde z_j, \Gamma) = d(z_j, \Gamma) + |z_j - \tilde z_j| > d(z_j, \Gamma) = r_j.
\end{equation}
Since $w_j \to 1$ in $C_{\rm loc}^{1, \alpha} (\R^n)$ for any $0 < \alpha < 1$, by \eqref{eq:utzj},
\begin{equation}\label{eq:tzj-zj}
|\tilde z_j - z_j| \bigg( \frac{f(u(z_j))}{u(z_j)} \bigg)^\frac{1}{2} \to + \infty ~~~~~~ \textmd{as} ~ j \to \infty.
\end{equation}
By \eqref{eq:fuzj} and \eqref{eq:dtzj}, we have
\begin{equation}\label{eq:fgft}
\frac{f(u(z_j))}{u(z_j)} \geq \frac{f(u(\tilde z_j))}{u(\tilde z_j)}.
\end{equation}
By \eqref{eq:t0-b}, \eqref{eq:uzj} and \eqref{eq:utzj}, we also have
\begin{equation}\label{eq:ftg2f}
\frac{f(u(\tilde z_j))}{u(\tilde z_j)} \geq \frac{2 f(u(z_j))}{u(z_j)} \bigg( \frac{u(\tilde z_j)}{u(z_j)} \bigg)^\frac{n + 2}{n - 2} \geq 2^{ - \frac{4}{n - 2} } \frac{f(u(z_j))}{u(z_j)}.
\end{equation}
This together with \eqref{eq:rj2fuzj} and \eqref{eq:dtzj} yields that
$$
d(\tilde z_j, \Gamma)^2 \frac{f(u(\tilde z_j))}{u(\tilde z_j)} \geq 2^{ - \frac{4}{n - 2} } d(z_j, \Gamma)^2 \frac{f(u(z_j))}{u(z_j)} \to + \infty ~~~~~~ \textmd{as} ~ j \to \infty.
$$

Let
$$
\tilde w_j (y) = \frac{1}{u(\tilde z_j)} u \bigg( \tilde z_j + \bigg( \frac{u(\tilde z_j)}{f(u(\tilde z_j))} \bigg)^\frac{1}{2} y \bigg) ~~~~~~ \textmd{for} ~ y \in \tilde\Sigma_j,
$$
where
$$
\tilde\Sigma_j = \bigg\{ y \in \R^n : \tilde z_j + \bigg( \frac{u(\tilde z_j)}{f(u(\tilde z_j))} \bigg)^\frac{1}{2} y \in B_2 \setminus \Gamma \bigg\}.
$$
As before, $\tilde w_j$ satisfies $\tilde w_j (0) = 1$ and
\begin{equation}\label{eq:twj-eq}
- \Delta \tilde w_j = \frac{f(u(\tilde z_j) \tilde w_j)}{f(u(\tilde z_j))} ~~~~~~ \textmd{in} ~ \tilde\Sigma_j.
\end{equation}
For any $R > 0$ and $y \in \bar B_R$, let $z = \tilde z_j + \big( \frac{u(\tilde z_j)}{f(u(\tilde z_j))} \big)^\frac{1}{2} y$. Then, by \eqref{eq:dtzj}-\eqref{eq:ftg2f},
$$
\aligned
d(z, \Gamma) & \geq d(\tilde z_j, \Gamma) - |\tilde z_j - z| \\
& \geq d(\tilde z_j, \Gamma) - \bigg( \frac{u(\tilde z_j)}{f(u(\tilde z_j))} \bigg)^\frac{1}{2} |y| \\
& \geq d(z_j, \Gamma) + \bigg( \frac{u(\tilde z_j)}{f(u(\tilde z_j))} \bigg)^\frac{1}{2} \bigg[ |z_j - \tilde z_j| \bigg( \frac{f(u(\tilde z_j))}{u(\tilde z_j)} \bigg)^\frac{1}{2} - R \bigg] \\
& \geq d(z_j, \Gamma)
\endaligned
$$
for large $j$. Thus, by \eqref{eq:fuzj} and \eqref{eq:ftg2f},
$$
\frac{f(u(z))}{u(z)} \leq \frac{f(u(z_j))}{u(z_j)} \leq 2^\frac{4}{n - 2} \frac{f(u(\tilde z_j))}{u(\tilde z_j)},
$$
which implies that
$$
0 \leq \frac{f(u(\tilde z_j) \tilde w_j)}{f(u(\tilde z_j)) \tilde w_j} \leq 2^\frac{4}{n - 2} ~~~~~~ \textmd{in} ~ \bar B_R
$$
for large $j$. Therefore, after passing to a subsequence, $\tilde w_j$ converges to a positive function $\tilde w$ in $C_{\rm loc}^{1, \alpha} (\R^n)$ for any $0 < \alpha < 1$. We assert that $\tilde w \not\equiv 1$ in $\R^n$.

Assume $\tilde w \equiv 1$ in $\R^n$. Then by \eqref{eq:utzj}, $u(z_j) \geq u(z) \geq u(z_j)/4$ for large $j$ and $y \in \bar B_1$, where $z = \tilde z_j + \big( \frac{u(\tilde z_j)}{f(u(\tilde z_j))} \big)^\frac{1}{2} y$. Thus, by \eqref{eq:t0-b}, \eqref{eq:uzj} and \eqref{eq:fgft}, we have for $y \in \bar B_1$,
$$
\aligned
\frac{f(u(\tilde z_j) \tilde w_j (y))}{f(u(\tilde z_j))} = \frac{f(u(z))}{f(u(\tilde z_j))} & \geq \frac{f(u(z_j))}{f(u(\tilde z_j))} \bigg( \frac{u(z)}{u(z_j)} \bigg)^\frac{n + 2}{n - 2} \\
& \geq \frac{u(z_j)}{u(\tilde z_j)} \bigg( \frac{u(z)}{u(z_j)} \bigg)^\frac{n + 2}{n - 2} \geq 4^{ - \frac{n + 2}{n - 2} }.
\endaligned
$$
Integrating \eqref{eq:twj-eq}, we get
$$
- \int_{\partial B_1} \partial_\nu \tilde w_j \,{\rm d}S = \int_{B_1} \frac{f(u(\tilde z_j) \tilde w_j (y))}{f(u(\tilde z_j))} \,{\rm d}y \geq 4^{ - \frac{n + 2}{n - 2} } |B_1|,
$$
which reaches a contradicton, as $\tilde w \equiv 1$ in $\R^n$, the left hand side tends to $0$ as $j \to \infty$. Claim 1 is proved. Consequently, we can always assume that $w \not\equiv 1$ in $\R^n$.

Since $w \not\equiv 1$ in $\R^n$, there exists a constant $c_0 > 0$ independent of $j$ such that
$$
\int_{B_1} \frac{f(u(z_j) w_j (y))}{f(u(z_j))} \,{\rm d}y \geq c_0 > 0.
$$
By changing variables, we have
\begin{equation}\label{eq:intcv}
\bigg( \frac{f(u(z_j))}{ u(z_j)^\frac{n}{n - 2} } \bigg)^\frac{n - 2}{2} \int_{B_{ (u(z_j)/f(u(z_j)))^{1/2} } (z_j)} f(u(z)) \,{\rm d}z \geq c_0 > 0.
\end{equation}
By \cite[Lemma 2.3]{CLin01}, we know that $f(u) \in L^1 (B_1)$. Thus, by \eqref{eq:rj2fuzj} and \eqref{eq:intcv}, we obtain
\begin{equation}\label{eq:uzjsub}
\lim_{j \to \infty} u(z_j)^{ - \frac{n}{n - 2} } f(u(z_j)) = + \infty.
\end{equation}

On the other hand, we will show that, for every $x \in \R^n$ and $\lambda > 0$,
\begin{equation}\label{eq:moving-w}
w(y) \geq w_{x, \lambda} (y) ~~~~~~ \forall\, y \in \R^n \setminus B_\lambda (x),
\end{equation}
where $w_{x, \lambda}$ is the Kelvin transform of $w$ with respect to the sphere $\partial B_\lambda (x)$, defined as in \eqref{eq:uK}. By \cite[Lemma 11.2]{LZhang03}, \eqref{eq:moving-w} implies that $w \equiv 1$ in $\R^n$. This contadicts Claim 1.

Let us fix $x_0 \in \R^n$ and $\lambda_0 > 0$ arbitrarily. Then for all $j$ large, we have
$$
|x_0| < \frac{R_j}{10} ~~~~~~ \textmd{and} ~~~~~~ \lambda_0 < \frac{R_j}{10}.
$$
Denote
$$
M_j := \bigg( \frac{f(u(z_j))}{u(z_j)} \bigg)^\frac{1}{2}
$$
and
$$
\Gamma_j := \Big\{ y \in \R^n : z_j + M_j^{- 1} y \in \Gamma \Big\}.
$$
We are going to show that for all sufficiently large $j$,
\begin{equation}\label{eq:moving-wj}
w_j (y) \geq (w_j)_{x_0, \lambda_0} (y) ~~~~~~ \forall\, y \in B_{M_j} (x_0) \setminus (B_{\lambda_0} (x_0) \cup \Gamma_j).
\end{equation}
Then \eqref{eq:moving-w} follows from \eqref{eq:moving-wj} by sending $j \to \infty$. Notice that for large $j$, we always have $\Gamma_j \subset B_{M_j} (x_0)$ and $B_{\lambda_0} (x_0) \subset B_{M_j} (x_0) \setminus \Gamma_j \subset \Sigma_j$.

\medskip

\noindent{\bf Claim 2:} There exists a real number $\lambda_2 \in (0, \lambda_0)$ independent of (large) $j$ such that for any $0 < \lambda < \lambda_2$, we have
\begin{equation}\label{eq:claim2-b}
w_j (y) \geq (w_j)_{x_0, \lambda} (y) ~~~~~~ \forall\, y \in B_{M_j} (x_0) \setminus (B_\lambda (x_0) \cup \Gamma_j).
\end{equation}
The proof of Claim 2 consists of two steps.

\medskip

{\it Step 1.} We show that there exists $\lambda_1 \in (0, \lambda_0)$ independent of (large) $j$ such that for any $0 < \lambda < \lambda_1$,
$$
w_j (y) \geq (w_j)_{x_0, \lambda} (y) ~~~~~~ \forall\, y \in B_{\lambda_1} (x_0) \setminus B_\lambda (x_0).
$$
Since $w_j \to w$ in $C_{\rm loc}^{1, \alpha} (\R^n)$ for any $0 < \alpha < 1$ and $w$ is positive in $\R^n$, we know that $|\nabla \ln w_j| \leq C_0$ in $B_1 (x_0)$ for all $j$ sufficiently large. Then we have
\begin{equation}\label{eq:drwj}
\aligned
\frac{ {\rm d} }{ {\rm d}r } \Big( r^\frac{n - 2}{2} w_j (x_0 + r \theta) \Big) & = r^\frac{n - 4}{2} w_j (x_0 + r \theta) \bigg( \frac{n - 2}{2} + \frac{\nabla w_j \cdot \theta}{w_j} r \bigg) \\
& \geq r^\frac{n - 4}{2} w_j (x_0 + r \theta) \bigg( \frac{n - 2}{2} - C_0 r \bigg) > 0
\endaligned
\end{equation}
for all $0 < r < \lambda_1 := \min \big\{ 1, \frac{n - 2}{2 C_0}, \frac{\lambda_0}{2} \big\}$ and $\theta \in \Sp^{n - 1}$. For any $0 < \lambda < \lambda_1$ and $y \in B_{\lambda_1} (x_0) \setminus B_\lambda (x_0)$, let $\theta = \frac{y - x_0}{|y - x_0|}$, $r_1 = |y - x_0|$ and $r_2 = \frac{\lambda^2 r_1}{|y - x_0|^2}$. Using \eqref{eq:drwj} we have
$$
r_1^\frac{n - 2}{2} w_j (x_0 + r_1 \theta) \geq r_2^\frac{n - 2}{2} w_j (x_0 + r_2 \theta).
$$
That is, for any $0 < \lambda < \lambda_1$,
\begin{equation}\label{eq:la1-b}
w_j (y) \geq (w_j)_{x_0, \lambda} (y) ~~~~~~ \forall\, y \in B_{\lambda_1} (x_0) \setminus B_\lambda (x_0).
\end{equation}

{\it Step 2.} We show that there exists $\lambda_2 \in (0, \lambda_1]$ independent of (large) $j$ such that \eqref{eq:claim2-b} holds for all $0 < \lambda < \lambda_2$.

\medskip

Let $\Phi_j (y) = \big( \frac{\lambda_1}{|y - x_0|} \big)^{n - 2} \inf_{\partial B_{\lambda_1} (x_0)} w_j$. Then
$$
- \Delta \Phi_j = 0 ~~~~~~ \textmd{in} ~ \R^n \setminus B_{\lambda_1} (x_0).
$$
Obviously, $w_j \geq \Phi_j$ on $\partial B_{\lambda_1} (x_0)$. Now we examine them on the boundary $\partial B_{M_j} (x_0)$. Since $u \geq \inf_{\partial B_2} u > 0$ in $\bar B_2 \setminus \Gamma$, it follows that for $y \in \partial B_{M_j} (x_0)$,
$$
w_j (y) \geq \frac{\inf_{\partial B_2} u}{u(z_j)} > 0.
$$
Thus, by \eqref{eq:uzjsub}, for $y \in \partial B_{M_j} (x_0)$ and large $j$,
\begin{equation}\label{eq:Phijwj}
\bigg( \frac{\lambda_1}{|y - x_0|} \bigg)^{n - 2} \inf_{\partial B_{\lambda_1} (x_0)} w_j \leq C \bigg( \frac{f(u(z_j))}{ u(z_j)^\frac{n}{n - 2} } \bigg)^{ - \frac{n - 2}{2} } \frac{1}{u(z_j)} \ll \frac{\inf_{\partial B_2} u}{u(z_j)} \leq w_j (y),
\end{equation}
where we used \eqref{eq:Harnack} in the first inequality. Notice that $(w_j - \Phi_j)$ is superharmonic and bounded from below in $B_{M_j} (x_0) \setminus (B_{\lambda_1} (x_0) \cup \Gamma_j)$, by Lemma \ref{lem:suphar-int},
$$
w_j (y) \geq \bigg( \frac{\lambda_1}{|y - x_0|} \bigg)^{n - 2} \inf_{\partial B_{\lambda_1} (x_0)} w_j > 0 ~~~~~~ \forall\, y \in B_{M_j} (x_0) \setminus (B_{\lambda_1} (x_0) \cup \Gamma_j).
$$
Let
$$
\lambda_2 := \lambda_1 \bigg( \frac{\inf_{\partial B_{\lambda_1} (x_0)} w}{2 \sup_{B_{\lambda_1} (x_0)} w} \bigg)^\frac{1}{n - 2} \leq \lambda_1 \bigg( \frac{\inf_{\partial B_{\lambda_1} (x_0)} w_j}{\sup_{B_{\lambda_1} (x_0)} w_j} \bigg)^\frac{1}{n - 2} \leq \lambda_1.
$$
Then for any $0 < \lambda < \lambda_2$ and $y \in B_{M_j} (x_0) \setminus (B_{\lambda_1} (x_0) \cup \Gamma_j)$, we have
$$
\aligned
(w_j)_{x_0, \lambda} (y) & = \bigg( \frac{\lambda}{|y - x_0|} \bigg)^{n - 2} w_j (y^{x_0, \lambda}) \\
& \leq \bigg( \frac{\lambda_2}{|y - x_0|} \bigg)^{n - 2} \sup_{B_{\lambda_1} (x)} w_j \\
& \leq \bigg( \frac{\lambda_1}{|y - x_0|} \bigg)^{n - 2} \inf_{\partial B_{\lambda_1} (x_0)} w_j \leq w_j (y).
\endaligned
$$
This together with \eqref{eq:la1-b} leads to Claim 2.

Now we define
$$
\bar\lambda_j := \sup \{ \mu \in (0, \lambda_0) : w_j (y) \geq (w_j)_{x_0, \lambda} (y),~ \forall\, 0 < \lambda < \mu,~ y \in B_{M_j} (x_0) \setminus (B_\lambda (x_0) \cup \Gamma_j) \},
$$
where $\lambda_0$ is fixed at the beginning. By Claim 2, $\bar\lambda_j$ is well-defined and $0 < \lambda_2 \leq \bar\lambda_j \leq \lambda_0$ for all sufficiently large $j$.

\medskip

\noindent{\bf Claim 3:} $\bar\lambda_j = \lambda_0$ for all sufficiently large $j$.

\medskip

Suppose $\bar\lambda_j < \lambda_0$. By the definition of $\bar\lambda_j$,
\begin{equation}\label{eq:bala-b}
w_j (y) \geq (w_j)_{x_0, \bar\lambda_j} (y) ~~~~~~ \forall\, y \in B_{M_j} (x_0) \setminus (B_{\bar\lambda_j} (x_0) \cup \Gamma_j).
\end{equation}
Furthermore, we wish to prove
\begin{equation}\label{eq:strong-b}
w_j (y) > (w_j)_{x_0, \bar\lambda_j} (y) ~~~~~~ \forall\, y \in \bar B_{M_j} (x_0) \setminus (\bar B_{\bar\lambda_j} (x_0) \cup \Gamma_j).
\end{equation}
First, by \eqref{eq:uzjsub}, for $y \in \partial B_{M_j} (x_0)$ and large $j$, we have
\begin{equation}\label{eq:pBMj}
\aligned
(w_j)_{x_0, \bar\lambda_j} (y) & = \bigg( \frac{\bar\lambda_j}{|y - x_0|} \bigg)^{n - 2} w_j (y^{x_0, \bar\lambda_j}) \\
& \leq \bigg( \frac{\lambda_0}{|y - x_0|} \bigg)^{n - 2} \sup_{\bar B_{\lambda_0} (x_0)} w_j \\
& \leq C \bigg( \frac{f(u(z_j))}{ u(z_j)^\frac{n}{n - 2} } \bigg)^{ - \frac{n - 2}{2} } \frac{1}{u(z_j)} \\
& \ll \frac{\inf_{\partial B_2} u}{u(z_j)} \leq w_j (y)
\endaligned
\end{equation}
as the proof of \eqref{eq:Phijwj}.

Let
\begin{equation}\label{eq:Omj}
\Omega_j := \Big\{ y \in B_{M_j} (x_0) \setminus (B_{\bar\lambda_j} (x_0) \cup \Gamma_j) : w_j (y) < w_j (y^{x_0, \bar\lambda_j}) \Big\}.
\end{equation}
Clearly, $\Omega_j$ is an open subset of $B_{M_j} (x_0) \setminus (\bar B_{\bar\lambda_j} (x_0) \cup \Gamma_j)$. Notice that for $y \in \Omega_j$ and large $j$,
\begin{equation}\label{eq:wjyK-C}
w_j (y^{x_0, \bar\lambda_j}) \geq \inf_{B_{\bar\lambda_j} (x_0)} w_j \geq \inf_{B_{\lambda_0} (x_0)} w_j \geq \frac{1}{C(n, |x_0| + \lambda_0)} > 0,
\end{equation}
where we used \eqref{eq:Harnack} in the last inequality. By \eqref{eq:uzj} and \eqref{eq:wjyK-C}, we always have $u(z_j) w_j (y^{x_0, \bar\lambda_j}) \gg t_0 + 1$ for $y \in \Omega_j$ and large $j$.

For $y \in \Omega_j$, we have
\begin{equation}\label{eq:sup-b}
\aligned
& - \Delta (w_j - (w_j)_{x_0, \bar\lambda_j}) \\
& = \frac{f(u(z_j) w_j)}{f(u(z_j))} - \bigg( \frac{\bar\lambda_j}{|y - x_0|} \bigg)^{n + 2} \frac{f(u(z_j) w_j (y^{x_0, \bar\lambda_j}))}{f(u(z_j))} \\
& = (u(z_j) w_j)^{ - \frac{n + 2}{n - 2} } \frac{f(u(z_j) w_j)}{f(u(z_j))} (u(z_j) w_j)^\frac{n + 2}{n - 2} - \bigg( \frac{\bar\lambda_j}{|y - x_0|} \bigg)^{n + 2} \frac{f(u(z_j) w_j (y^{x_0, \bar\lambda_j}))}{f(u(z_j))}.
\endaligned
\end{equation}
In order to apply the monotone assumption \eqref{eq:t0-b} in the above, we need to compare the values of $t_0$ and $u(z_j) w_j (y)$ for $y \in \Omega_j$. We shall divide the discussion into two subregions: $\Omega_j'$ and $\Omega_j''$. For convenience, we assume that $t_0 > \inf_{\partial B_2} u > 0$ in \eqref{eq:t0-b} (If $\inf_{\partial B_2} u \geq t_0 \geq 0$, we only have the subregion $\Omega_j'$). Let the subregion $\Omega_j'$ be defined by
$$
\Omega_j' := \{ y \in \Omega_j : u(z_j) w_j (y) \geq t_0 \}.
$$
For $y \in \Omega_j'$, by \eqref{eq:t0-b} and \eqref{eq:bala-b}, we can write \eqref{eq:sup-b} to
\begin{equation}\label{eq:sup-b1}
\aligned
& - \Delta (w_j - (w_j)_{x_0, \bar\lambda_j}) \\
& \geq \bigg[ \bigg( \frac{w_j (y)}{w_j (y^{x_0, \bar\lambda_j})} \bigg)^\frac{n + 2}{n - 2} - \bigg( \frac{\bar\lambda_j}{|y - x_0|} \bigg)^{n + 2} \bigg] \frac{f(u(z_j) w_j (y^{x_0, \bar\lambda_j}))}{f(u(z_j))} \\
& \geq \bigg( \frac{\bar\lambda_j}{|y - x_0|} \bigg)^{n + 2} \bigg[ \bigg( \frac{w_j (y)}{(w_j)_{x_0, \bar\lambda_j} (y)} \bigg)^\frac{n + 2}{n - 2} - 1 \bigg] \frac{f(u(z_j) w_j (y^{x_0, \bar\lambda_j}))}{f(u(z_j))} \geq 0.
\endaligned
\end{equation}
Let the subregion $\Omega_j''$ be defined by
$$
\Omega_j'' := \{ y \in \Omega_j : u(z_j) w_j (y) < t_0 \}.
$$
Then, for $y \in \Omega_j''$ we distinguish two cases.

\medskip

{\it Case 1.} $\lim\nolimits_{t \to + \infty} t^{ - \frac{n + 2}{n - 2} } f(t) = 0$. By \eqref{eq:uzj}, \eqref{eq:bala-b} and \eqref{eq:wjyK-C}, we have
\begin{equation}\label{eq:sup-b2}
\aligned
& - \Delta (w_j - (w_j)_{x_0, \bar\lambda_j}) \\
& \geq \frac{ (u(z_j) w_j)^\frac{n + 2}{n - 2} }{f(u(z_j))} \inf_{[\inf_{\partial B_2} u,\, t_0]} \frac{f(t)}{ t^\frac{n + 2}{n - 2} } - \bigg( \frac{\bar\lambda_j}{|y - x_0|} \bigg)^{n + 2} \frac{(u(z_j) w_j (y^{x_0, \bar\lambda_j}))^\frac{n + 2}{n - 2} o(1)}{f(u(z_j))} \\
& \geq \frac{ (u(z_j) w_j)^\frac{n + 2}{n - 2} }{f(u(z_j))} \bigg( \inf_{[\inf_{\partial B_2} u,\, t_0]} \frac{f(t)}{ t^\frac{n + 2}{n - 2} } - o(1) \bigg) \geq 0.
\endaligned
\end{equation}

{\it Case 2.} $\lim\nolimits_{t \to + \infty} t^{ - \frac{n + 2}{n - 2} } f(t) > 0$. By \eqref{eq:uzj}, \eqref{eq:Harnack} and \eqref{eq:wjyK-C}, we have
\begin{equation}\label{eq:sup-b3}
\aligned
& - \Delta (w_j - (w_j)_{x_0, \bar\lambda_j}) \\
& \geq - \bigg( \frac{\bar\lambda_j}{|y - x_0|} \bigg)^{n + 2} \frac{f(u(z_j) w_j (y^{x_0, \bar\lambda_j}))}{f(u(z_j))} \\
& \geq - \bigg( \frac{\bar\lambda_j}{|y - x_0|} \bigg)^{n + 2} \frac{ (u(z_j) w_j (y^{x_0, \bar\lambda_j}))^\frac{n + 2}{n - 2} }{ u(z_j)^\frac{n + 2}{n - 2} } (1 + o(1)) \\
& \geq - \bigg( \frac{\bar\lambda_j}{|y - x_0|} \bigg)^{n + 2} w_j (y^{x_0, \bar\lambda_j})^\frac{n + 2}{n - 2} (1 + o(1)) \geq - \frac{C_1}{ |y - x_0|^{n + 2} }
\endaligned
\end{equation}
for some constant $C_1 > 0$ independent of $j$. Recall that $\bar\lambda_j \geq \lambda_2 > 0$, where the latter is independent of $j$. Then, we have
$$
\aligned
t_0 > u(z_j) w_j (y) & \geq u(z_j) \bigg( \frac{\bar\lambda_j}{|y - x_0|} \bigg)^{n - 2} w_j (y^{x_0, \bar\lambda_j}) \\
& \geq u(z_j) \bigg( \frac{\lambda_2}{|y - x_0|} \bigg)^{n - 2} \frac{1}{C(n, |x_0| + \lambda_0)} > 0.
\endaligned
$$
Thus, $\Omega_j''$ is an open subset of a spherical shell
$$
S_j := \bigg\{ x \in \R^n : \frac{1}{C_2} u(z_j)^\frac{1}{n - 2} < |x - x_0| < M_j \bigg\} \subset B_{M_j} (x_0) \setminus B_{\bar\lambda_j} (x_0)
$$
for some constant $C_2 > 0$ independent of $j$.

Set
$$
\varphi_j (y) := \frac{\inf_{\partial B_2} u}{2 u(z_j)} \bigg[ 1 - \bigg( \frac{\bar\lambda_j}{|y - x_0|} \bigg)^{n - 2} \bigg] - C_1 \int_{S_j} \frac{G(y, x; x_0, \bar\lambda_j)}{ |y - x_0|^{n + 2} } \,{\rm d}x
$$
for $y \in B_{M_j} (x_0) \setminus B_{\bar\lambda_j} (x_0)$, where
$$
G(y, x; x_0, \bar\lambda_j) = \frac{1}{n (n - 2) |B_1|} \bigg[ \frac{1}{ |y - x|^{n - 2} } - \bigg( \frac{\bar\lambda_j}{|y - x_0|} \bigg)^{n - 2} \frac{1}{ |y^{x_0, \bar\lambda_j} - x|^{n - 2} } \bigg].
$$
A straightforward calculation gives for $y \in B_{M_j} (x_0) \setminus B_{\bar\lambda_j} (x_0)$,
$$
0 \leq \int_{S_j} \frac{G(y, x; x_0, \bar\lambda_j)}{ |y - x_0|^{n + 2} } \,{\rm d}x \leq \frac{C}{ u(z_j)^\frac{n}{n - 2} } \bigg[ 1 - \bigg( \frac{\bar\lambda_j}{|y - x_0|} \bigg)^{n - 2} \bigg].
$$
Hence, for $y \in B_{M_j} (x_0) \setminus \bar B_{\bar\lambda_j} (x_0)$ and sufficiently large $j$,
\begin{equation}\label{eq:phij-est}
0 < \frac{\inf_{\partial B_2} u}{4 u(z_j)} \bigg[ 1 - \bigg( \frac{\bar\lambda_j}{|y - x_0|} \bigg)^{n - 2} \bigg] \leq \varphi_j (y) \leq \frac{\inf_{\partial B_2} u}{2 u(z_j)} \bigg[ 1 - \bigg( \frac{\bar\lambda_j}{|y - x_0|} \bigg)^{n - 2} \bigg].
\end{equation}

By \eqref{eq:sup-b1}-\eqref{eq:sup-b3} and the definition of $\varphi_j$,
$$
- \Delta (w_j - (w_j)_{x_0, \bar\lambda_j} - \varphi_j) \geq 0 ~~~~~~ \textmd{in} ~ \Omega_j.
$$
Now we estimate $(w_j - (w_j)_{x_0, \bar\lambda_j} - \varphi_j)$ on $\partial\Omega_j \setminus \Gamma_j$. If $y \in \partial B_{\bar\lambda_j} (x_0)$, then
$$
(w_j - (w_j)_{x_0, \bar\lambda_j} - \varphi_j) (y) = 0.
$$
If $y \in \partial B_{M_j} (x_0)$, by \eqref{eq:pBMj} and \eqref{eq:phij-est} we obtain
$$
\aligned
& (w_j - (w_j)_{x_0, \bar\lambda_j} - \varphi_j) (y) \\
& \geq \frac{\inf_{\partial B_2} u}{u(z_j)} - C \bigg( \frac{f(u(z_j))}{ u(z_j)^\frac{n}{n - 2} } \bigg)^{ - \frac{n - 2}{2} } \frac{1}{u(z_j)} - \frac{\inf_{\partial B_2} u}{2 u(z_j)} > 0.
\endaligned
$$
Similarly, if $y \in B_{M_j} (x_0) \setminus (\bar B_{\lambda_j} (x_0) \cup \Gamma_j)$ and $w_j (y) = w_j (y^{x_0, \bar\lambda_j})$, then
$$
\aligned
& (w_j - (w_j)_{x_0, \bar\lambda_j} - \varphi_j) (y) \\
& \geq \bigg[ 1 - \bigg( \frac{\bar\lambda_j}{|y - x_0|} \bigg)^{n - 2} \bigg] w_j (y^{x_0, \bar\lambda_j}) - \frac{\inf_{\partial B_2} u}{2 u(z_j)} \bigg[ 1 - \bigg( \frac{\bar\lambda_j}{|y - x_0|} \bigg)^{n - 2} \bigg] \\
& \geq \bigg[ 1 - \bigg( \frac{\bar\lambda_j}{|y - x_0|} \bigg)^{n - 2} \bigg] \bigg( \frac{\inf_{\partial B_2} u}{u(z_j)} - \frac{\inf_{\partial B_2} u}{2 u(z_j)} \bigg) > 0.
\endaligned
$$
By Lemma \ref{lem:suphar-int}, we have
\begin{equation}\label{eq:wjphij}
(w_j - (w_j)_{x_0, \bar\lambda_j}) (y) \geq \varphi_j (y) > 0 ~~~~~~ \forall\, y \in \Omega_j.
\end{equation}
Clearly, \eqref{eq:wjphij} holds automatically in $(\bar B_{M_j} (x_0) \setminus (\bar B_{\bar\lambda_j} (x_0) \cup \Gamma_j)) \setminus \Omega_j$. Inequality \eqref{eq:strong-b} is proved.

We also wish to show
\begin{equation}\label{eq:hopf-b}
\partial_\nu (w_j - (w_j)_{x_0, \bar\lambda_j}) (y) > 0 ~~~~~~ \forall\, y \in \partial B_{\bar\lambda_j} (x_0),
\end{equation}
where $\nu$ is the unit outward normal vector on $\partial B_{\bar\lambda_j} (x_0)$. For $z_0 \in \partial B_{\bar\lambda_j} (x_0)$ satisfying $\partial_\nu (w_j - (w_j)_{x_0, \bar\lambda_j}) (z_0) < (n - 2) \bar\lambda_j^{- 1} w_j (z_0)$, by a direct calculation
$$
(n - 2) \bar\lambda_j^{- 1} w_j (z_0) > \partial_\nu (w_j - (w_j)_{x_0, \bar\lambda_j}) (z_0) = (n - 2) \bar\lambda_j^{- 1} w_j (z_0) + 2 \partial_\nu w_j(z_0).
$$
Thus, $\partial_\nu w_j (z_0) < 0$ and
$$
\partial_\nu (w_j (\cdot^{x_0, \bar\lambda_j}) - w_j) (z_0) = - 2 \partial_\nu w_j (z_0) > 0.
$$
Recall that $w_j = (w_j)_{x_0, \bar\lambda_j} = w_j (\cdot^{x_0, \bar\lambda_j})$ on $\partial B_{\bar\lambda_j} (x_0)$. Then there exists a small $\delta > 0$ such that $B_\delta (z_0) \setminus \bar B_{\bar\lambda_j} (x_0) \subset \Omega_j'$. By \eqref{eq:strong-b}, \eqref{eq:sup-b1} and the Hopf lemma, $\partial_\nu (w_j - (w_j)_{x_0, \bar\lambda_j}) (z_0) > 0$.

By \eqref{eq:hopf-b} and the compactness of $\partial B_{\bar\lambda_j} (x_0)$, we have
$$
\partial_\nu (w_j - (w_j)_{x_0, \bar\lambda_j}) \big|_{\partial B_{\bar\lambda_j} (x_0)} \geq b > 0.
$$
By the continuity of $w_j$ and $\nabla w_j$, there exists $R \in (\bar\lambda_j, \lambda_0)$ such that for any $\bar\lambda_j \leq \lambda \leq r \leq R$,
$$
\partial_\nu (w_j - (w_j)_{x_0, \bar\lambda_j}) \big|_{\partial B_r (x_0)} \geq \frac{b}{2} > 0.
$$
Consequently, since $w_j = (w_j)_{x_0, \lambda}$ on $\partial B_\lambda (x_0)$, we have for any $\bar\lambda_j \leq \lambda < R$,
\begin{equation}\label{eq:BR-b}
w_j (y) > (w_j)_{x_0, \lambda} (y) ~~~~~~ \forall\, y \in B_R (x_0) \setminus \bar B_\lambda (x_0).
\end{equation}
By \eqref{eq:phij-est} and \eqref{eq:wjphij}, for $y \in B_{M_j} (x_0) \setminus (B_R (x_0) \cup \Gamma_j)$ and large $j$, we have
$$
(w_j - (w_j)_{x_0, \bar\lambda_j}) (y) \geq \frac{\inf_{\partial B_2} u}{4 u(z_j)} \bigg[ 1 - \bigg( \frac{\bar\lambda_j}{|y - x_0|} \bigg)^{n - 2} \bigg] \geq \frac{\inf_{\partial B_2} u}{4 u(z_j)} \bigg[ 1 - \bigg( \frac{\bar\lambda_j}{R} \bigg)^{n - 2} \bigg] > 0.
$$
Therefore, for $y \in B_{M_j} (x_0) \setminus (B_R (x_0) \cup \Gamma_j)$,
\begin{equation}\label{eq:wjc-b2}
(w_j - (w_j)_{x_0, \lambda}) (y) \geq \frac{\inf_{\partial B_2} u}{4 u(z_j)} \bigg[ 1 - \bigg( \frac{\bar\lambda_j}{R} \bigg)^{n - 2} \bigg] - ((w_j)_{x_0, \lambda} - (w_j)_{x_0, \bar\lambda_j}) (y).
\end{equation}
By the uniform continuity of $w_j$ on $\bar B_R (x_0)$, there exists $0 < \varepsilon_j \leq R - \bar\lambda_j$ such that for all $\bar\lambda_j \leq \lambda < \bar\lambda_j + \varepsilon_j$ and $y \in B_{M_j} (x_0) \setminus (B_R (x_0) \cup \Gamma_j)$,
$$
\bigg| \lambda^{n - 2} w_j \bigg( x_0 + \frac{\lambda^2 (y - x_0)}{|y - x_0|^2} \bigg) - \bar\lambda_j^{n - 2} w_j \bigg( x_0 + \frac{\bar\lambda_j^2 (y - x_0)}{|y - x_0|^2} \bigg) \bigg| \leq \frac{R^{n - 2} \inf_{\partial B_2} u}{8 u(z_j)} \bigg[ 1 - \bigg( \frac{\bar\lambda_j}{R} \bigg)^{n - 2} \bigg].
$$
This together with \eqref{eq:wjc-b2} implies that for all $\bar\lambda_j \leq \lambda < \bar\lambda_j + \varepsilon_j$,
$$
w_j (y) > (w_j)_{x_0, \lambda} (y) ~~~~~~ \forall\, y \in B_{M_j} (x_0) \setminus (B_R (x_0) \cup \Gamma_j).
$$
This and \eqref{eq:BR-b} contradict the definition of $\bar\lambda_j$. We finished the proof of Claim 3. Recalling the limit analysis before Claim 2 (See \eqref{eq:moving-w} and \eqref{eq:moving-wj}), the proof of Proposition \ref{prop:blowup} is completed.
\end{proof}

Now, we use an improved moving sphere method and Proposition \ref{prop:blowup} to prove Theorem \ref{thm:asym}.

\begin{proof}[Proof of Theorem \ref{thm:asym}] Without loss of generality, we may assume that $\Omega = B_2$, $\Gamma \subset B_{1/2}$, $u \in C^2 (\bar B_2 \setminus \Gamma)$ and $f$ is continuous on $(0, + \infty)$. Because $f$ is positive, we know that $u$ is superharmonic in $B_2 \setminus \Gamma$. By Lemma \ref{lem:suphar-int}, $u \geq \inf_{\partial B_2} u > 0$ in $\bar B_2 \setminus \Gamma$. We assert that
\begin{equation}\label{eq:dn-2u-0}
\lim_{d(x, \Gamma) \to 0} d(x, \Gamma)^{n - 2} u(x) = 0.
\end{equation}
We prove \eqref{eq:dn-2u-0} by contradiction. Suppose there exists a constant $\varepsilon_0 > 0$ and a sequence $\{ x_i \} \subset B_{1/2} \setminus \Gamma$ such that $d(x_i, \Gamma) \to 0$ as $i \to \infty$, but
\begin{equation}\label{eq:dn-2u-N}
d(x_i, \Gamma)^{n - 2} u(x_i) > \varepsilon_0 ~~~~~~ \textmd{for all} ~ i \in \N_+.
\end{equation}
Then, by Proposition \ref{prop:blowup} and \eqref{eq:dn-2u-N} we have
$$
\aligned
\frac{f(u(x_i))}{ u(x_i)^\frac{n}{n - 2} } & \leq \frac{f(u(x_i))}{u(x_i)} u(x_i)^{ - \frac{2}{n - 2} } \\
& \leq C d(x_i, \Gamma)^{- 2} (\varepsilon_0 d(x_i, \Gamma)^{2 - n})^{ - \frac{2}{n - 2} } \leq C \varepsilon_0^{ - \frac{2}{n - 2} },
\endaligned
$$
which is contradicted by \eqref{eq:grow} since $u(x_i) \to + \infty$ as $i \to \infty$. Thus, the limit \eqref{eq:dn-2u-0} holds.

We will show that there exists a small constant $\rho \in (0, d(\Gamma, \partial B_{1/2})/2)$ such that for every $x$ satisfying $0 < d(x, \Gamma) \leq \rho$ and $\lambda \in (0, d(x, \Gamma))$,
\begin{equation}\label{eq:moving-a}
u(y) \geq u_{x, \lambda} (y) ~~~~~~ \forall\, y \in B_1 \setminus (B_\lambda (x) \cup \Gamma),
\end{equation}
where $u_{x, \lambda}$ is defined as in \eqref{eq:uK}.

First of all, for every $x$ satisfying $0 < d(x, \Gamma) < d(\Gamma, \partial B_{1/2})/2$, since $u \in C^2 (\bar B_2 \setminus \Gamma)$ is positive, we can suppose
$$
|\nabla \ln u| \leq C_0 ~~~~~~ \textmd{in} ~ B_{d(x, \Gamma)/2} (x)
$$
for some constant $C_0 > 0$. Then we have
\begin{equation}\label{eq:dru-a}
\aligned
\frac{ {\rm d} }{ {\rm d}r } \Big( r^\frac{n - 2}{2} u(x + r \theta) \Big) & = r^\frac{n - 4}{2} u(x + r \theta) \bigg( \frac{n - 2}{2} + \frac{\nabla u \cdot \theta}{u} r \bigg) \\
& \geq r^\frac{n - 4}{2} u(x + r \theta) \bigg( \frac{n - 2}{2} - C_0 r \bigg) > 0
\endaligned
\end{equation}
for all $0 < r < \lambda_1 := \min\big\{ \frac{n - 2}{2 C_0}, \frac{d(x, \Gamma)}{2} \big\}$ and $\theta \in \Sp^{n - 1}$. For any $0 < \lambda < \lambda_1$ and $y \in B_{\lambda_1} (x) \setminus B_\lambda (x)$, let $\theta = \frac{y - x}{|y - x|}$, $r_1 = |y - x|$ and $r_2 = \frac{\lambda^2 r_1}{|y - x|^2}$. Using \eqref{eq:dru-a} we have
$$
r_1^\frac{n - 2}{2} u(x + r_1 \theta) \geq r_2^\frac{n - 2}{2} u(x + r_2 \theta).
$$
That is, for any $0 < \lambda < \lambda_1$,
\begin{equation}\label{eq:la1-a}
u(y) \geq u_{x, \lambda} (y) ~~~~~~ \forall\, y \in B_{\lambda_1} (x) \setminus B_\lambda (x).
\end{equation}
Furthermore, we may let $\lambda_1 > 0$ be even smaller such that for $y \in \partial B_1$,
$$
u(y) \geq \inf_{\partial B_2} u \geq (2 \lambda_1)^{n - 2} \sup_{B_{d(x, \Gamma)/2} (x)} u \geq \bigg( \frac{\lambda_1}{|y - x|} \bigg)^{n - 2} \inf_{B_{\lambda_1} (x)} u > 0.
$$
Obviously, the above inequality holds on $\partial B_{\lambda_1} (x)$. By Lemma \ref{lem:suphar-int}, we have
$$
u(y) \geq \bigg( \frac{\lambda_1}{|y - x|} \bigg)^{n - 2} \inf_{\partial B_{\lambda_1} (x)} u > 0 ~~~~~~ \forall\, y \in B_1 \setminus (B_{\lambda_1} (x) \cup \Gamma).
$$
Take
$$
\lambda_2 := \lambda_1 \bigg( \frac{\inf_{\partial B_{\lambda_1} (x)} u}{\sup_{B_{\lambda_1} (x) } u} \bigg)^\frac{1}{n - 2} \leq \lambda_1.
$$
Then for all $0 < \lambda < \lambda_2$ and $y \in B_1 \setminus (B_{\lambda_1} (x) \cup \Gamma)$,
$$
\aligned
u_{x, \lambda} (y) & = \bigg( \frac{\lambda}{|y - x|} \bigg)^{n - 2} u(y^{x, \lambda}) \\
& \leq \bigg( \frac{\lambda_2}{|y - x|} \bigg)^{n - 2} \sup_{B_{\lambda_1} (x)} u \\
& \leq \bigg( \frac{\lambda_1}{|y - x|} \bigg)^{n - 2} \inf_{\partial B_{\lambda_1} (x)} u \leq u(y).
\endaligned
$$
Combining \eqref{eq:la1-a} with the above, we obtain that for all $0 < \lambda < \lambda_2$,
$$
u(y) \geq u_{x, \lambda} (y) ~~~~~~ \forall\, y \in B_1 \setminus (B_\lambda (x) \cup \Gamma).
$$
Therefore,
$$
\bar\lambda (x) := \sup \{ \mu \in (0, d(x, \Gamma)) : u(y) \geq u_{x, \lambda} (y),~ \forall\, 0 < \lambda < \mu,~ y \in B_1 \setminus (B_\lambda (x) \cup \Gamma) \}
$$
is well-defined for any $x$ satisfying $0 < d(x, \Gamma) < (0, d(\Gamma, \partial B_{1/2})/2)$ and is positive.

Next we show that there exists a small constant $\rho \in (0, d(\Gamma, \partial B_{1/2})/2)$ such that $\bar\lambda (x) = d(x, \Gamma)$ for all $x$ satisfying $0 < d(x, \Gamma) \leq \rho$.

For $y \in \partial B_1$ and $0 < \lambda < d(x, \Gamma) < \frac{d(\Gamma, \partial B_{1/2})}{2} \leq \frac{1}{4}$, we have
$$
|y^{x, \lambda} - x| \leq 2 \lambda^2 \leq 2 d(x, \Gamma)^2 \leq \frac{d(x, \Gamma)}{2}
$$
and
$$
|y^{x, \lambda}| \leq |y^{x, \lambda} - x| + |x| \leq \frac{d(x, \Gamma)}{2} + |x| \leq \frac{d(\Gamma, \partial B_{1/2})}{4} + \bigg( \frac{1}{2} - \frac{d(\Gamma, \partial B_{1/2})}{2} \bigg) < \frac{1}{2}.
$$
Hence
$$
d(y^{x, \lambda}, \Gamma) \leq |y^{x, \lambda} - x| + d(x, \Gamma) \leq \frac{3 d(x, \Gamma)}{2}
$$
and
$$
d(y^{x, \lambda}, \Gamma) \geq d(x, \Gamma) - |y^{x, \lambda} - x| \geq \frac{d(x, \Gamma)}{2}.
$$
It follows from \eqref{eq:dn-2u-0} and the above estimates that there exists a sufficiently small $\rho > 0$ independent of $x$ such that for all $y \in \partial B_1$, $x$ satisfying $0 < d(x, \Gamma) \leq \rho$ and $0 < \lambda < d(x, \Gamma)$,
\begin{equation}\label{eq:pB1la}
\aligned
u_{x, \lambda} (y) & = \bigg( \frac{\lambda}{|y - x|} \bigg)^{n - 2} u(y^{x, \lambda}) \\
& \leq (2 d(x, \Gamma))^{n - 2} u(y^{x, \lambda}) \leq \frac{\inf_{\partial B_2} u}{2} < \inf_{\partial B_2} u \leq u(y).
\endaligned
\end{equation}

Suppose $\bar\lambda (x) < d(x, \Gamma)$ for some $x$ satisfying $0 < d(x, \Gamma) \leq \rho$. For brevity, we will denote $\bar\lambda (x) = \bar\lambda$ in the below. By the definition of $\bar\lambda$,
\begin{equation}\label{eq:bala-a}
u(y) \geq u_{x, \bar\lambda} (y) ~~~~~~ \forall\, y \in B_1 \setminus (B_{\bar\lambda} (x) \cup \Gamma).
\end{equation}
As the proof of \eqref{eq:strong-g} in Theorem \ref{thm:sym-x1}, we assert that
\begin{equation}\label{eq:strong-a}
u(y) > u_{x, \bar\lambda} (y) ~~~~~~ \forall\, y \in \bar B_1 \setminus (\bar B_{\bar\lambda} (x) \cup \Gamma).
\end{equation}

Let
\begin{equation}\label{eq:Om1-a}
\Omega_1 := \Big\{ y \in B_1 \setminus (\bar B_{\bar\lambda} (x) \cup \Gamma) : u(y) < u(y^{x, \bar\lambda}) \Big\}.
\end{equation}
Clearly, $\Omega_1$ is an open subset of $B_1 \setminus (\bar B_{\bar\lambda} (x) \cup \Gamma)$. By \eqref{asym-mono} and \eqref{eq:bala-a}, for $y \in \Omega_1$ we have
\begin{equation}\label{eq:sup-a}
\aligned
- \Delta (u - u_{x, \bar\lambda}) & = f(u) - \bigg( \frac{\bar\lambda}{|y - x|} \bigg)^{n + 2} f(u(y^{x, \bar\lambda})) \\
& \geq \bigg( \frac{u(y)}{u(y^{x, \bar\lambda})} \bigg)^\frac{n + 2}{n - 2} f(u(y^{x, \bar\lambda})) - \bigg( \frac{\bar\lambda}{|y - x|} \bigg)^{n + 2} f(u(y^{x, \bar\lambda})) \\
& \geq \bigg( \frac{\bar\lambda}{|y - x|} \bigg)^{n + 2} \bigg[ \bigg( \frac{u(y)}{u_{x, \bar\lambda} (y)} \bigg)^\frac{n + 2}{n - 2} - 1 \bigg] f(u(y^{x, \bar\lambda})) \geq 0.
\endaligned
\end{equation}
As the proof of \eqref{eq:strong-g} in Theorem \ref{thm:sym-x1}. Suppose there exists $y_0 \in B_1 \setminus (\bar B_{\bar\lambda} (x) \cup \Gamma)$ such that $u(y_0) = u_{x, \bar\lambda} (y_0)$, then we can show that $\{ y \in B_1 \setminus (\bar B_{\bar\lambda} (x) \cup \Gamma) : u(y) = u_{x, \bar\lambda} (y) \}$ is a both open and closed nonempty subset of $B_1 \setminus (\bar B_{\bar\lambda} (x) \cup \Gamma)$. Hence, $u \equiv u_{x, \bar\lambda}$ in $B_1 \setminus (\bar B_{\bar\lambda} (x) \cup \Gamma)$. This contradicts \eqref{eq:pB1la}. Inequality \eqref{eq:strong-a} is proved.

Next, we wish to show
\begin{equation}\label{eq:hopf-a}
\partial_\nu (u - u_{x, \bar\lambda}) (y) > 0 ~~~~~~ \forall\, y \in \partial B_{\bar\lambda} (x),
\end{equation}
where $\nu$ is the unit outward normal vector on $\partial B_{\bar\lambda} (x)$. For $z_0 \in \partial B_{\bar\lambda} (x)$ satisfying $\partial_\nu (u - u_{x, \bar\lambda}) (z_0) < (n - 2) \bar\lambda^{- 1} u(z_0)$, by a direct calculation,
$$
(n - 2) \bar\lambda^{- 1} u(z_0) > \partial_\nu (u - u_{x, \bar\lambda}) (z_0) = (n - 2) \bar\lambda^{- 1} u(z_0) + 2 \partial_\nu u(z_0).
$$
Thus, $\partial_\nu u(z_0) < 0$ and
$$
\partial_\nu (u(\cdot^{x, \bar\lambda}) - u) (z_0) = - 2 \partial_\nu u(z_0) > 0.
$$
Recall that $u = u_{x, \bar\lambda} = u(\cdot^{x, \bar\lambda})$ on $\partial B_{\bar\lambda} (x)$. Then there exists a small $\delta > 0$ such that $B_\delta (z_0) \setminus \bar B_{\bar\lambda} (x) \subset \Omega_1$. By \eqref{eq:strong-a}, \eqref{eq:sup-a} and the Hopf lemma,
$$
\partial_\nu (u - u_{x, \bar\lambda}) (z_0) > 0.
$$
By the compactness of $\partial B_{\bar\lambda} (x)$,
$$
\partial_\nu (u - u_{x, \bar\lambda}) \big|_{\partial B_{\bar\lambda} (x)} \geq b > 0.
$$
By the continuity of $u$ and $\nabla u$, there exists $R \in (\bar\lambda, d(x, \Gamma))$ such that for all $\bar\lambda \leq \lambda \leq r \leq R$,
$$
\partial_\nu (u - u_{x, \lambda}) \big|_{\partial B_r (x)} \geq \frac{b}{2} > 0.
$$
Since $u = u_{x, \lambda}$ on $\partial B_\lambda (x)$, we have that for every $\bar\lambda \leq \lambda < R$,
\begin{equation}\label{eq:BR-a}
u(y) > u_{x, \lambda} (y) ~~~~~~ \forall\, y \in B_R (x) \setminus \bar B_\lambda (x).
\end{equation}
As the proof of \eqref{eq:uC2G-g} in Theorem \ref{thm:sym-x1}, to apply Lemma \ref{lem:suphar-int} to $(u - u_{x, \bar\lambda})$ on $\Omega_1 \setminus \bar B_R (x)$, we need to estimate $(u - u_{x, \bar\lambda})$ on $\partial (\Omega_1 \setminus \bar B_R (x)) \setminus \Gamma$. If $y \in \partial B_R (x)$, then by \eqref{eq:strong-a},
$$
(u - u_{x, \bar\lambda}) (y) \geq \inf_{\partial B_R (x)} (u - u_{x, \bar\lambda}) > 0.
$$
If $y \in B_1 \setminus (\bar B_R (x) \cup \Gamma)$ and $u(y) = u(y^{x, \bar\lambda})$, then
$$
u(y) - u_{x, \bar\lambda} (y) \geq \bigg[ 1 - \bigg( \frac{\bar\lambda}{R} \bigg)^{n - 2} \bigg] \inf_{B_{\bar\lambda} (x)} u > 0.
$$
By \eqref{eq:pB1la} and Lemma \ref{lem:suphar-int},
\begin{equation}\label{eq:uCR-a}
(u - u_{x, \bar\lambda}) (y) \geq C(R) > 0 ~~~~~~ \forall\, y \in \Omega_1 \setminus B_R (x)
\end{equation}
for some constant
$$
0 < C(R) \leq \min \bigg\{ \inf_{\partial B_R (x)} (u - u_{x, \bar\lambda}), \bigg[ 1 - \bigg( \frac{\bar\lambda}{R} \bigg)^{n - 2} \bigg] \inf_{B_{\bar\lambda} (x)} u, \frac{\inf_{\partial B_2} u}{2} \bigg\}.
$$
Obviously, \eqref{eq:uCR-a} holds automatically in $(B_1 \setminus (B_R (x) \cup \Gamma)) \setminus \Omega_1$. Hence, there exists $0 < \varepsilon < R - \bar\lambda$ such that for every $\bar\lambda \leq \lambda < \bar\lambda + \varepsilon$ and $y \in B_1 \setminus (B_R (x) \cup \Gamma)$,
$$
|(u_{x, \lambda} - u_{x, \bar\lambda}) (y)| \leq \frac{1}{2} \inf_{B_1 \setminus (B_R (x) \cup \Gamma)} (u - u_{x, \bar\lambda}).
$$
Therefore, for every $\bar\lambda \leq \lambda < \bar\lambda + \varepsilon$ and $y \in B_1 \setminus (B_R (x) \cup \Gamma)$,
$$
\aligned
(u - u_{x, \lambda}) (y) & \geq (u - u_{x, \bar\lambda}) (y) - |(u_{x, \lambda} - u_{x, \bar\lambda}) (y)| \\
& \geq \frac{1}{2} (u - u_{x, \bar\lambda}) (y) > 0.
\endaligned
$$
This together with \eqref{eq:BR-a} leads to a contradiction with the definition of $\bar\lambda$. Consequently, inequality \eqref{eq:moving-a} holds.

Let $r > 0$ small (less than $\rho^2$), $x_1, x_2 \in \Pi_r^{- 1} (z)$ with $z \in \Gamma$ be such that
$$
u(x_1) = \max_{\Pi_r^{- 1} (z)} u, ~~~~~~ u(x_2) = \min_{\Pi_r^{- 1} (z)} u.
$$
Let $e_1 = x_1 - z$, $e_2 = x_2 - z$, $x_3 = x_1 + \frac{\rho (e_1 - e_2)}{4 |e_1- e_2|}$. Then $e_1, e_2 \in (T_z \Gamma)^\perp$ and thus, $e_2 - e_1 \in (T_z \Gamma)^\perp$. Let $\lambda = \sqrt{\frac{\rho}{4} \big( |e_1 - e_2| + \frac{\rho}{4} \big)}$. It is easy to check that $0 < \lambda < |x_3 - z| = d(x_3, \Gamma) < \rho$. From \eqref{eq:moving-a} we obtain
$$
u(x_2) \geq u_{x_3, \lambda} (x_2).
$$
On the other hand, the definition of $u_{x_3, \lambda}$ gives
$$
\aligned
u_{x_3, \lambda} (x_2) & = \bigg( \frac{\lambda}{|e_1 - e_2| + \rho/4} \bigg)^{n - 2} u(x_1) \\
& = \bigg( \frac{1}{4 |e_1 - e_2|/\rho + 1} \bigg)^\frac{n - 2}{2} u(x_1) \\
& \geq \bigg(\frac{1}{8 r/\rho + 1} \bigg)^\frac{n - 2}{2} u(x_1).
\endaligned
$$
Thus,
$$
\max_{\Pi_r^{- 1} (z)} u \leq (8 r/\rho + 1)^\frac{n - 2}{2} \min_{\Pi_r^{- 1} (z)} u.
$$
This implies that for all $x, y \in \Pi_r^{- 1} (z)$,
$$
u(x) = u(y) (1 + O(r)) ~~~~~~ \textmd{as} ~ r \to 0^+,
$$
where $O(r)$ is uniform for $z \in \Gamma$. The proof of Theorem \ref{thm:asym} is completed.
\end{proof}

Finally, we give the proof of Theorem \ref{thm:asym-0}.

\begin{proof}[Proof of Theorem \ref{thm:asym-0}] Without loss of generality, we may assume that $\Omega = B_2$, $u \in C^2 (\bar B_2 \setminus \{ 0 \})$ and $f$ is continuous on $(0, + \infty)$. By assumption \eqref{eq:mono-0}, there is a constant $t_0 \geq 0$ such that
\begin{equation}\label{eq:t0-0}
t^{ - \frac{n + 2}{n - 2} } f(t) ~ \textmd{is nonincreasing for} ~ t \geq t_0.
\end{equation}
If the origin is a removable singularity, then there is nothing to prove. If the origin is non-removable, we claim that
\begin{equation}\label{eq:uinfty}
u(x) \to + \infty ~~~~~~ \textmd{as} ~ x \to 0.
\end{equation}
For $r \in (0, 1]$, we set $v_r (x) = u(rx)$ for $x \in B_2 \setminus B_{1/2}$. Then $v_r$ satisfies
$$
\Delta v_r (x) + C_r (x) v_r (x) = 0 ~~~~~~ \textmd{for} ~ x \in B_2 \setminus B_{1/2},
$$
where $C_r (x) := r^2 \frac{f(v_r (x))}{v_r (x)}$. By Proposition \ref{prop:blowup}, we have $|C_r (x)| \leq C |x|^{- 2}$ for $x \in B_2 \setminus B_{1/2}$. Applying the standard Harnack inequality yields
$$
\max_{|x| = 1} v_r (x) \leq C \min_{|x| = 1} v_r (x).
$$
Scaling back to $u$, we obtain the spherical Harnack inequality
\begin{equation}\label{eq:sphHar}
\max_{|x| = r} u(x) \leq C \min_{|x| = r} u(x) ~~~~~~ \textmd{for} ~ 0 < r \leq 1.
\end{equation}
Since the origin is a non-removable singularity, there exists a sequence of points $\{ x_i \} \subset B_1 \setminus \{ 0 \}$ such that $r_i = |x_i| \downarrow 0$ and $u(x_i) \uparrow + \infty$ as $i \to \infty$. By the maximum principle and the spherical Harnack inequality \eqref{eq:sphHar},
$$
\inf_{r_{i + 1} \leq |x| \leq r_i} u(x) \geq \inf_{ |x| = r_i, r_{i + 1} } u(x) \geq C^{- 1} \min (u(x_i), u(x_{i + 1})) \to + \infty.
$$
Claim \eqref{eq:uinfty} is proved. Thus, there exists a constant $r_0 \in (0, 2)$ such that
\begin{equation}\label{eq:r0}
u(x) > t_0 ~~~~~~ \textmd{for all} ~ x \in B_{r_0} \setminus \{ 0 \}.
\end{equation}

For $s \in [1, + \infty)$, we define the spherical average of $u$ on the $(n - 1)$-dimensional sphere of radius $s^{1/(2 - n)}$ centered at the origin
$$
A(s) = \frac{ s^\frac{n - 1}{n - 2} }{n |B_1|} \int_{ \partial B_{ s^{1/(2 - n)} } } u \,{\rm d}S.
$$
Then we have
$$
A''(s) = - \frac{ s^{ - \frac{n - 1}{n - 2} } }{n (n - 2)^2 |B_1|} \int_{ \partial B_{ s^{1/(2 - n)} } } f(u) \,{\rm d}S \leq 0.
$$
It follows that $A(s)$ is a positive concave $C^2$ function on $(1, + \infty)$. Thus, $A'(t)$ is nonnegative and nonincreasing on $(1, + \infty)$. Therefore,
$$
\lim_{s \to + \infty} \frac{A(s)}{s} = \lim_{s \to + \infty} A'(s) ~~~ \textmd{exists and is nonnegative}.
$$
This implies that the limit
\begin{equation}\label{eq:xn-2u-C0}
\lim_{x \to 0} |x|^{n - 2} \bar u(x) = C_0 ~~~ \textmd{exists and is nonnegative}.
\end{equation}

\noindent{\bf Case 1. $C_0 > 0$.} By the spherical Harnack inequality \eqref{eq:sphHar}, there exist two constants $C_1, C_2 > 0$ such that
$$
C_1 |x|^{2 - n} \leq u(x) \leq C_2 |x|^{2 - n} ~~~ \textmd{near the origin}.
$$
By \eqref{eq:t0-0} and the above estimate, we get
$$
f(u(x)) \leq \bigg( \frac{u(x)}{ C_1 |x|^{2 - n} } \bigg)^\frac{n + 2}{n - 2} f(C_1 |x|^{2 - n}) \leq \bigg( \frac{C_2}{C_1} \bigg)^\frac{n + 2}{n - 2} f(C_1 |x|^{2 - n})
$$
and
$$
f(u(x)) \geq \bigg( \frac{u(x)}{ C_2 |x|^{2 - n} } \bigg)^\frac{n + 2}{n - 2} f(C_2 |x|^{2 - n}) \geq \bigg( \frac{C_1}{C_2} \bigg)^\frac{n + 2}{n - 2} f(C_2 |x|^{2 - n}) \geq 0
$$
near the origin. Since $f(u(x)) \in L_{\rm loc}^1 (B_2)$, we have $f(C_1 |x|^{2 - n}), f(C_2 |x|^{2 - n}) \in L_{\rm loc}^1 (B_2)$. On the other hand, we define
$$
v(x) = \frac{1}{n (n - 2) |B_1|} \int_{B_1} \frac{f(u(y))}{ |x - y|^{n - 2} } \,{\rm d}y ~~~~~~ \textmd{for} ~ x \in B_1 \setminus \{ 0 \}.
$$
Then
$$
\aligned
v(x) & \leq C \int_{B_1} \frac{f(C_1 |y|^{2 - n})}{ |x - y|^{n - 2} } \,{\rm d}y \\
& \leq C \int_0^1 r^{n - 1} f(C_1 r^{2 - n}) \bigg( \int_{ \Sp^{n - 1} } \frac{1}{ |x - r \theta|^{n - 2} } \,{\rm d}S \bigg) \,{\rm d}r \\
& \leq C \int_0^1 r^{n - 1} f(C_1 r^{2 - n}) \frac{1}{ \max\{ |x|, r \}^{n - 2} } \,{\rm d}r \\
& \leq C |x|^{2 - n} \int_0^{|x|} r^{n - 1} f(C_1 r^{2 - n}) \,{\rm d}r + C \int_{|x|}^1 r f(C_1 r^{2 - n}) \,{\rm d}r.
\endaligned
$$
For the second term, since $f(C_1 |x|^{2 - n}) \in L_{\rm loc}^1 (B_2)$, for any $\varepsilon > 0$, there exists $\delta > 0$ such that
$$
\int_0^\delta r^{n - 1} f(C_1 r^{2 - n}) \,{\rm d}r < \varepsilon.
$$
Then, for $x \in B_\delta \setminus \{ 0 \}$, we have
$$
\aligned
\int_{|x|}^1 r f(C_1 r^{2 - n}) \,{\rm d}r & \leq \int_{|x|}^\delta r f(C_1 r^{2 - n}) \,{\rm d}r + \int_\delta^1 r f(C_1 r^{2 - n}) \,{\rm d}r \\
& \leq \varepsilon |x|^{2 - n} + C_\delta.
\endaligned
$$
Therefore, $\limsup_{x \to 0} |x|^{n - 2} \int_{|x|}^1 r f(C_1 r^{2 - n}) \,{\rm d}r \leq \varepsilon$. Since $\varepsilon > 0$ is arbitrary, we obtain $\lim_{x \to 0} |x|^{n - 2} \int_{|x|}^1 r f(C_1 r^{2 - n}) \,{\rm d}r = 0$. Consequently,
$$
v(x) = o(|x|^{2 - n}) ~~~ \textmd{near the origin}.
$$
It follows that $(u - v)$ is a positive harmonic function in $B_{r_1} \setminus \{ 0 \}$ for some small $r_1 > 0$. By the classical B\^ocher's theorem, there exists a constant $C > 0$ and a harmonic function $h \in C^\infty (B_{r_1})$ such that
$$
u(x) = \frac{C}{ |x|^{n - 2} } + \frac{1}{n (n - 2) |B_1|} \int_{B_1} \frac{f(u(y))}{ |x - y|^{n - 2} } \,{\rm d}y + h(x)
$$
near the origin. By \eqref{eq:xn-2u-C0}, we have $C = C_0$ and
$$
\lim_{x \to 0} |x|^{n - 2} u(x) = C_0.
$$

\noindent{\bf Case 2. $C_0 = 0$.} By the spherical Harnack inequality \eqref{eq:sphHar}, we have
\begin{equation}\label{eq:u-oG}
\lim_{x \to 0} |x|^{n - 2} u(x) = 0.
\end{equation}
This implies that the corresponding \eqref{eq:dn-2u-0} holds with $\Gamma = \{ 0 \}$. Noting further \eqref{eq:t0-0} and \eqref{eq:r0}, the moving sphere method in Theorem \ref{thm:asym} can be carried out in the same way. Therefore, we can find a small constant $\rho \in (0, r_0/4)$ such that for every $x \in \bar B_\rho \setminus \{ 0 \}$ and $\lambda \in (0, |x|)$,
$$
u(y) \geq u_{x, \lambda} (y) ~~~~~~ \forall\, y \in B_{r_0/2} \setminus (B_\lambda (x) \cup \{ 0 \}).
$$
This yields the asymptotic radial symmetry of $u$ near the origin. The proof of Theorem \ref{thm:asym-0} is completed.
\end{proof}

\noindent{\bf Conflict of interest}: On behalf of all authors, the corresponding author states that there is no conflict of interest.

\vskip0.10in

\noindent{\bf Data availability}: This manuscript has no associated data.

\bigskip

\noindent Xusheng Du

\noindent School of Mathematical Sciences \\
Shenzhen University, Shenzhen 518061, P. R. China \\[1mm]
Email: \textsf{xushengdu@szu.edu.cn}

\bigskip

\noindent Hui Yang

\noindent School of Mathematical Sciences \\
Shanghai Jiao Tong University, Shanghai 200240, P. R. China \\[1mm]
Email: \textsf{hui-yang@sjtu.edu.cn}

\end{document}